\documentclass[10pt]{amsart}

\usepackage{amsthm, amsfonts, amssymb, color}
 \usepackage{mathrsfs}
\usepackage{amsmath}
 \usepackage{amstext, amsxtra}
  \usepackage{txfonts}
 \usepackage[colorlinks, linkcolor=black, citecolor=blue, pagebackref, hypertexnames=false]{hyperref}
\usepackage{graphicx}
\usepackage{overpic}

 \allowdisplaybreaks
\newtheorem{theorem}{Theorem}[section]
\newtheorem{proposition}[theorem]{Proposition}
\newtheorem{coro}[theorem]{Corollary}
\newtheorem{lemma}[theorem]{Lemma}
\newtheorem{definition}[theorem]{Definition}

\newtheorem{remark}[theorem]{Remark}

\renewcommand\Re{\operatorname{Re}}
\renewcommand\Im{\operatorname{Im}}

\title[Scattering and blow-up criteria for INLS with a potential ]{
Scattering and blow-up criteria for  3D cubic focusing nonlinear inhomogeneous NLS with a potential }
\author{Qing Guo, \ Hua Wang \ and Xiaohua Yao}

\address {Qing Guo, College of Science,
 Minzu University of China, Beijing, 100081, P.R. China}
\email{guoqing0117@163.com}
\address{ Hua Wang, School of Mathematics and Statistics and Hubei Province Key Laboratory of Mathematical Physics,
Central China Normal University, Wuhan, 430079, P.R. China}
\email{wanghua\_math@126.com}
\address{Xiaohua Yao, School of Mathematics and Statistics and  Hubei Province Key Laboratory of Mathematical Physics,
 Central China Normal University, Wuhan, 430079, P.R. China}
\email{yaoxiaohua@mail.ccnu.edu.cn }
\date{\today}
\subjclass[2000]{ 35B20; 35P25; 35Q55; 47J35}
\keywords{Inhomogeneous NLS; Focusing; Potential; Scattering; Blow-up.}

\begin{document}

\maketitle
\begin{abstract}
  In this paper, we consider the 3d cubic focusing inhomogeneous nonlinear Schr\"{o}dinger equation with a potential
  $$
  iu_{t}+\Delta u-Vu+|x|^{-b}|u|^{2}u=0,\;\;(t,x) \in {{\bf{R}}\times{\bf{R}}^{3}},
  $$
  where $0<b<1$. We first establish global well-posedness and scattering for the radial initial data $u_{0}$ in $H^{1}({\bf R}^{3})$
  satisfying $M(u_{0})^{1-s_{c}}E(u_{0})^{s_{c}}<\mathcal{E}$ and $\|u_{0}\|_{L^{2}}^{2(1-s_{c})}\|H^{\frac{1}{2}}u_{0}\|_{L^{2}}^{2s_{c}}<\mathcal{K}$
  provided that $V$
  is repulsive, where $\mathcal{E}$ and $\mathcal{K}$ are the mass-energy and mass-kinetic of the ground states, respectively. Our result extends the results of Hong \cite{H}
  and Farah-Guzm$\acute{\rm a}$n \cite{FG1} with $b\in(0,\frac12)$ to the case $0<b<1$. We then obtain a blow-up result for initial data
  $u_{0}$ in $H^{1}({\bf R}^{3})$
  satisfying $M(u_{0})^{1-s_{c}}E(u_{0})^{s_{c}}<\mathcal{E}$ and $\|u_{0}\|_{L^{2}}^{2(1-s_{c})}\|H^{\frac{1}{2}}u_{0}\|_{L^{2}}^{2s_{c}}>\mathcal{K}$ if $V$ satisfies some
  additional assumptions.
\end{abstract}

\tableofcontents

\section{Introduction}
\setcounter{equation}{0}
In this paper, we consider a 3d cubic focusing inhomogeneous NLS with a potential ($\rm{INLS_{V}}$)
\begin{equation}\label{1.1}
\left\{ \begin{aligned}
  i&u_{t}-Hu+|x|^{-b}|u|^{2}u=0,\;\;(t,x) \in {{\bf{R}}\times{\bf{R}}^{3}}, \\
  u&(0, x)=u_{0}(x)\in H^{1}({\bf{R}}^{3}),
 \end{aligned}\right.
 \end{equation}
where $u: I\times {\bf R}^{3}\rightarrow {\bf C}$ is a complex-valued function, $0<b<1$, $H=H_{0}+V$,
$H_{0}=-\Delta$. Here $V: {\bf R}^{3}\rightarrow {\bf R}$ is a real-valued short range potential with a small negative part,
more precisely,
\begin{align}\label{1.2}
V\in K_{0}\cap L^{\frac{3}{2}}
\end{align}
and
\begin{align}\label{1.3}
\|V_{-}\|_{K}<4\pi,
\end{align}
where the potential class $K_{0}$ is the closure of bounded compactly supported functions with respect to the
global Kato norm
\begin{align*}
\|V\|_{K}\triangleq \sup_{x\in {\bf R}^{3}}\displaystyle\int_{{\bf R}^{3}}\frac{|V(y)|}{|x-y|}dy
\end{align*}
and $V_{-}(x)=\min\{V(x), 0\}$ is the negative part of $V$. In the case $V=0$ and $b=0$, Holmer-Roudenko \cite{HR} and Duyckaerts-Homer-Roudenko \cite{DHR}
obtained the sharp criteria for global well-posedness and scattering in terms of conservation laws of the equation \eqref{1.1},
where blow up result requires initial data is radial.
Subsequently, for $b=0$, Hong \cite{H} established an analogous global well-posed and scattering result provided that $V$ satisfies \eqref{1.2}
and \eqref{1.3}, $V\geq 0$, $x\cdot \nabla V\leq 0$ and $|x||\nabla V|\in L^{\frac{3}{2}}$. However, he cannot give any blow up result.
More recently, for $V=0$, Farah-Guz$\acute{\rm a}$n \cite{FG1} and Dinh \cite{Dinh3}
extended the scattering result and the blow up reslut obtained by Holmer-Roudenko \cite{HR} to $0<b<\frac{1}{2}$ and $0<b<1$ under the radial assumption on the initial data
$u_{0}$, respectively.

The mainly part of this paper is devoted to get a similar criteria for global well-posedness and scattering for \eqref{1.1} with the radial data $u_{0}$
under the similar condition on $V$ as that in \cite{H} and over the wider interval $0<b<1$.
Additionally, we further give a non-scattering or blow-up result based on the method
of Du-Wu-Zhang \cite{DWZ} under some additional assumptions on $V$.

Before the statement of our results, we briefly review some related results for the general $\rm{INLS_{V}}$ equation
\begin{equation}\label{1.4}
\left\{ \begin{aligned}
  i&u_{t}-Hu+g(x)|u|^{p-1}u=0,\;\;(t,x) \in {{\bf{R}}\times{\bf{R}}^{N}}, \\
  u&(0, x)=u_{0}(x),
 \end{aligned}\right.
 \end{equation}
For $p=1+\frac{4}{N}$, several authors have investigated critical mass blow-up solutions. For example, Banica-Carles-Duyckaerts \cite{BCD} showed the
existence of critical mass blow up solutions if $V\in C^{\infty}({\bf R}^{N}, {\bf R})$ and $g\in C^{\infty}({\bf R}^{N}, {\bf R})$ is
sufficiently flat at a critical point. When $V\equiv 0$ and $g\in C^{\infty}({\bf R}^{N}, {\bf R})$ is positive and bounded, Merle \cite{M} and
Rapha$\ddot{\rm e}$l-Szeflel \cite{RS} derived conditions on g for existence/nonexistence of minimal mass blow-up solutions. In the above works,
$V(x)$ and $g(x)$ are both smooth. While Combet-Genoud \cite{CG} studied the classification of minimal mass blow-up solutions in the case $V\equiv0$
and $g(x)=|x|^{-b}$ with $0<b<\min \{2, N\}$, $N\geq 1$. Besides, when $V(x)=-\frac{c}{|x|^{2}}$ with $0<c<\frac{(N-2)^{2}}{4}$, $N\geq 3$, Csobo-Genoud \cite{CG1}
constructed and classified finite time blow-up solutions at the minimal mass threshold.

Next we recall some well-posedness and scattering results for $V(x)\equiv0$ and $g(x)=|x|^{-b}$ with $0<b<\min\{2, N\}$. One can easily see that
the equation \eqref{1.4} is invariant under the scaling transformation $u(t, x)=\lambda^{\frac{2-b}{p-1}}u(\lambda^{2}t, \lambda x)$, which also
leaves the norm of the homogeneous Sobolev space $\dot{H}^{s_{c}}({\bf {R}}^{N})$ invariant, where $s_c=\frac{N}{2}-\frac{2-b}{p-1}$. So we call that
the equation \eqref{1.4} is mass-supercritical and energy-subcritical for
$1+p_{*}<p<1+p^*$ (i.e., $0<s_{c}<1$), where
\begin{equation*}
p^{*}=\left\{
\begin{aligned}
\frac{4-2b}{N-2}, N&\geq 3,\\
\infty, N&=1,2,
\end{aligned}
\right.
\quad\quad p_{*}=\frac{4-2b}{N}.
\end{equation*}
Energy-criticality appears with the power $p=1+\frac{4-2b}{N-2}$ (i.e., $s_{c}=1$) and mass-criticality with power $p=1+\frac{4-2b}{N}$ (i.e., $s_{c}=0$).
Genoud-Stuart \cite{GS}, using the abstract theory developed by Cazenave \cite{Caz}, showed that \eqref{1.4} is locally well-posed in $H^{1}({\bf R}^{N})$ if $1<p<1+p^{*}$ and
globally if $1<p<1+p_{*}$ for any initial data and $1+p_{*}\leq p<1+p^*$ for small initial data. Recently, Guzm$\acute{\rm a}$n \cite{Guz} gave an alternative proof of
these results using the contraction mapping principle based on the Strichartz estimates.

When $p=1+p_{*}$, Genoud \cite{Gen} showed that \eqref{1.4} is global well-posed in $H^{1}({\bf R}^{N})$ if $u_{0}\in H^{1}({\bf R}^{N})$ and
$$
\|u_{0}\|_{L^{2}}<\|Q_{m}\|_{L^{2}},
$$
where $Q_{m}$ is the ground state solution of the nonlinear elliptic equation
$$
\Delta Q_{m}-Q_{m}+|x|^{-b}|Q_{m}|^{\frac{4-2b}{N}}Q_{m}=0.
$$
On the other hand, Combet and Genoud \cite{CG} obtained the classification of minimal mass blow-up solutions for \eqref{1.4} with $p=1+p_{*}$.
When $1+p_*<p<1+p^*$, Farah \cite{Far} proved that \eqref{1.4} is globally well-posed in $H^{1}({\bf R}^{N})$, $N\geq 3$, assuming that
$u_{0}\in H^{1}({\bf R}^{N})$,
\begin{equation}\label{1.5}
E_{0}(u_{0})^{s_{c}}M(u_{0})^{1-s_{c}}<E_{0}(Q)^{s_{c}}M(Q)^{1-s_{c}},
\end{equation}
and
\begin{align}\label{1.6}
\|\nabla u_{0}\|_{L^{2}}^{2s_{c}}\|u_{0}\|_{L^{2}}^{2(1-s_{c})}<\|\nabla Q\|_{L^{2}}^{2s_{c}}\|Q\|_{L^{2}}^{2(1-s_{c})},
\end{align}
where $E_{0}$ and $M$ are a functional in \eqref{1.8} and \eqref{1.9}, and $Q$ is unique positive radial solution of the elliptic equation
\begin{equation}\label{1.7}
\Delta Q-Q+|x|^{-b}|Q|^{p-1}Q=0.
\end{equation}
Farah \cite{Far} also considers the case
\begin{align*}
\|\nabla u_{0}\|_{L^{2}}^{2s_{c}}\|u_{0}\|_{L^{2}}^{2(1-s_{c})}>\|\nabla Q\|_{L^{2}}^{2s_{c}}\|Q\|_{L^{2}}^{2(1-s_{c})},
\end{align*}
which combined  with \eqref{1.5} implies that the solution blows up in finite time if $u_{0}$ satisfies $|x|u_{0}\in L^{2}$.
In the radial case for $u_{0}$, Dinh \cite{Dinh3} removed the the condition $|x|u_{0}\in L^{2}$.

Moreover, Farah-Guzm$\acute{\rm a}$n \cite{FG1, FG2} established scattering in the case that $1+p_{*}<p<1+2^{*}$, $0<b<\min\{\frac{N}{3}, 1\}$ and $u_{0}$ is radial,
where
\begin{equation*}
2^{*}=\left\{
\begin{aligned}
p^{*},\quad N&=2, N\geq 4,\\
3-2b, N&=3.
\end{aligned}
\right.
\end{equation*}
We note that, for $N=3$, the authors imposed an extra assumption, namely, $1+p_{*}<p<1+(3-2b)$ (when $p=3$, $0<b<\frac{1}{2}$). Then they raised
a question whether scattering holds under the condition $1+p_{*}<p<1+p^{*}=1+(4-2b)$. One purpose of this present paper is to give an affirmative answer to this
question when $N=3$ and $p=3$. More precisely, we prove that scattering is true when $0<b<1$ (see Remark \ref{rem5.6}). The main ingredient is Lemma \ref{lem2.3}. In \cite{FG1, FG2},
in order to control nonlinear terms, the authors divide ${\bf R}^{3}$ into two parts: $|x|\leq 1$ and $|x|>1$. Different from them, we mainly utilize the 
Sobolev inequality Lemma \ref{lem2.4} to deal with nonlinear terms, which will make the proof more simple. 

When $b=0$, $V(x)$ is inverse-square potential, Killip-Murply-Visan-Zheng \cite{KMVZ} established the sharp criteria for the global well-posedness and scattering in
terms of conservation laws of \eqref{1.1}. The authors \cite{GWY} extended their results to the case $V(x)=\frac{k}{|x|^{\alpha}}$ with $k>0$ and $1<\alpha< 2$ by 
adopting the variational method of Ibrahim-Masmoudi-Nakanishi \cite{IMN}. 
And Hong \cite{H} established a similar result for real-valued short range and repulsive potential for the equation
\eqref{1.1} as mentioned above. In view of these results, we are further aimed at extending  Hong's result to $0<b<1$. 

Under the assumptions \eqref{1.2} and \eqref{1.3}, the Cauchy problem for \eqref{1.1} is locally well-posed in $H^{1}({\bf R}^{3})$ (see local theory Lemma \ref{lem2.2} ).
Moreover, the $H^{1}$ solution obeys the mass and energy conservation laws,
\begin{align}\label{1.8}
M(u)=\displaystyle\int_{{\bf R}^{3}}|u(x)|^{2}dx,
\end{align}
and the energy is defined by
\begin{align}\label{1.9}
E(u)=E_{V}(u)=\frac{1}{2}\displaystyle\int_{{\bf R}^{3}}|\nabla u(x)|^{2}dx
+\frac{1}{2}\displaystyle\int_{{\bf R}^{3}}V(x)|u(x)|^{2}dx
-\frac{1}{4}\displaystyle\int_{{\bf R}^{3}}|x|^{-b}|u(x)|^{4}dx.
\end{align}
When $V$ vanishes, we just replace $E(u)$ by $E_{0}(u)$.

To state our main results, we need to
introduce some notation as follows:
\begin{equation*}
\mathcal{E}=\left\{
\begin{aligned}
M(Q)^{1-s_{c}}E_{0}(Q)^{s_{c}}, \text{if}& \; V\geq 0,\\
M(\mathcal{Q})^{1-s_{c}}E(\mathcal{Q})^{s_{c}}, \text{if}& \; V\leq 0,
\end{aligned}
\right.
\end{equation*}

\begin{equation*}
\mathcal{K}=\left\{
\begin{aligned}
\|Q\|_{L^{2}}^{2(1-s_{c})}\|\nabla Q\|_{L^{2}}^{2s_{c}}, \text{if}& \; V\geq 0,\\
\|\mathcal{Q}\|_{L^{2}}^{2(1-s_{c})}\|H^{\frac{1}{2}}\mathcal{Q}\|_{L^{2}}^{2s_{c}}, \text{if}& \; V\leq 0,
\end{aligned}
\right.
\end{equation*}
where $Q$ is the ground state for the elliptic equation \eqref{1.7} with $p=3$ and $\mathcal{Q}$ solves the
elliptic equation $$(-\Delta +V)\mathcal{Q}+w_{\mathcal{Q}}^{2}\mathcal{Q}-|x|^{-b}|\mathcal{Q}|^{2}\mathcal{Q}=0, \;
w_{\mathcal{Q}}=\frac{\sqrt{1-b}}{\sqrt{3+b}}\frac{\|H^{\frac{1}{2}}\mathcal{Q}\|_{L^{2}}}{\|\mathcal{Q}\|_{L^{2}}}
$$
(see Lemma \ref{lem3.5} for details).
It follows from Remark \ref{rem3.2} and Lemma \ref{lem3.5} in section 3 and \eqref{4.6}, \eqref{4.7}, \eqref{4.9} and \eqref{4.10} in section 4 that
\begin{align*}
\mathcal{E}
=(\frac{s_{c}}{3+b})^{s_{c}}\mathcal{K}\quad\text{and}\quad\mathcal{K}=\frac{4}{(3+b)C_{GN}}=\frac{4}{(3+b)C_{GN}^{rad}},
\end{align*}
where $C_{GN}$ and $C_{GN}^{rad}$ are the sharp constants in the Gagaliardo-Nirenberg inequalities with the potential $V$, respectively.
It is worth pointing out that under out assumption \eqref{1.2}, $C_{GN}=C_{GN}^{rad}$;
while, in \cite{KMVZ}, $C_{GN}^{rad}<C_{GN}$
when $V(x)=\frac{a}{|x|^{2}}$ with $a>0$.
On the other hand, we will see in section 3 that $C_{GN}=C_{GN}^{rad}$ never be attained when $V$ is nonnegative and not zero a.e. on $\mathbb R^3$,
which is another different phenomenon from the inverse-square-potential case ($V(x)=\frac{a}{|x|^{2}}$ with $a>0$);
while $C_{GN}=C_{GN}^{rad}$ can be reached by $\mathcal Q$ when $V_-\neq0$. (One can find more details in section 3.)

Our first result provides criteria for global well-posedness in terms of the mass-energy $\mathcal{E}$ and
a critical number $\mathcal{K}$, which is involved with the kinetic energy.

\begin{theorem}\label{th1.1}
Suppose that $V$ is radially symmetric and satisfies \eqref{1.2} and \eqref{1.3}, and $0<b<1$. We assume that
\begin{align}\label{1.10}
M(u_{0})^{1-s_{c}}E(u_{0})^{s_{c}}<\mathcal{E}.
\end{align}
Let $u(t)$ be the solution to \eqref{1.1} with initial data $u_{0}\in H^{1}({\bf R}^{3})$.

(i) If
\begin{align}\label{1.11}
\|u_{0}\|_{L^{2}}^{2(1-s_{c})}\|H^{\frac{1}{2}}u_{0}\|_{L^{2}}^{2s_{c}}<\mathcal{K},
\end{align}
then $u(t)$ exists globally in time, and
\begin{align}\label{1.12}
\|u_{0}\|_{L^{2}}^{2(1-s_{c})}\|H^{\frac{1}{2}}u(t)\|_{L^{2}}^{2s_{c}}<\mathcal{K}, \forall t\in{\bf R}.
\end{align}

(ii) If
\begin{align}\label{1.13}
\|u_{0}\|_{L^{2}}^{2(1-s_{c})}\|H^{\frac{1}{2}}u_{0}\|_{L^{2}}^{2s_{c}}>\mathcal{K},
\end{align}
then
\begin{align}\label{1.14}
\|u_{0}\|_{L^{2}}^{2(1-s_{c})}\|H^{\frac{1}{2}}u(t)\|_{L^{2}}^{2s_{c}}>\mathcal{K}
\end{align}
during the maximal existence time.
\end{theorem}

\begin{remark}\label{Global}
(i) Theorem \ref{th1.1} also holds provided that nonnegative $V$ satisfies
\begin{align}\label{V1}
V,\; \nabla V\in L^{\delta}+L^{\infty}
\end{align}
for some $\delta\geq\frac{3}{2}$, or
\begin{align}\label{V2}
V\in L^{\delta}+L^{\infty}
\end{align}
for some $\delta>\frac{3}{2}$.
If $V$ satisfies \eqref{V1}, then the local wellposedness
is true by Remark 4.4.8 in \cite{Caz}, Remark 2) on page 103 in \cite{Sch} and Corollary 1.6 in \cite{Guz}, where the contraction mapping principle based on
the Strichartz estimates is used. If $V$ satisfies \eqref{V2}, then the local wellposedness is true by Theorem 4.3.1 in \cite{Caz} and Theorem K.1 and Lemma K.2 in
\cite{GS}, where the abstract theory developed by Cazenave is used.

(ii) The radial condition on $V$ in Theorem \ref{th1.1} is only used in the case that the initial data $u_{0}$ is radial, which will be applied to the
 following scattering result.
\end{remark}

Another result is to show that the global solutions in Theorem \ref{th1.1} also scatters provided that $u_{0}$ is radial,
$V$ is repulsive and $0<b<1$.

\begin{theorem}\label{th1.2}
Let $V$ be radially symmetric and  satisfy \eqref{1.2} and \eqref{1.3}, and assume that $x\cdot\nabla V(x)\leq 0$, $|x||\nabla V|\in L^{\frac{3}{2}}$
and $0<b<1$. If $u_{0}$ is radial data in $H^{1}({\bf R}^{3})$ and satisfies
$$
M(u_{0})^{1-s_{c}}E(u_{0})^{s_{c}}<\mathcal{E}
$$
and
$$
\|u_{0}\|_{L^{2}}^{2(1-s_{c})}\|H^{\frac{1}{2}}u_{0}\|_{L^{2}}^{2s_{c}}<\mathcal{K}
$$
then $u(t)$ scatters in $H^{1}({\bf R}^{3})$. That is, there exists $\phi_{\pm}\in H^{1}({\bf R}^{3})$ such that
\begin{equation*}
\lim_{t\rightarrow \pm\infty}\|u(t)-e^{it\Delta}\phi_{\pm}\|_{H^{1}({\bf R}^{3})}=0.
\end{equation*}
\end{theorem}
\vskip0.5cm
\begin{remark}\label{rem1.3}
(i) If $V$ is radial and $x\cdot\nabla V(x)\leq 0$, then $V$ is decreasing. Also as $V\in L^{\frac{3}{2}}$, $V(x)\rightarrow 0$ as $|x|\rightarrow+\infty$.
So $V\geq 0$. 

(ii) In the defocusing case and without potentials, Dinh \cite{Dinh2} obtained scattering in $H^{1}({\bf R}^{N})$ provided that $N\geq 4$, $0<b<2$,
$1+p_*<p<1+p^*$, or $N=3$, $0<b<1$, $1+\frac{5-2b}{3}<p<1+(3-2b)$, or $N=2$, $0<b<1$, $1+p_*<p<1+p^*$. It is easy to see that
when $N=3$ and $p=3$, $b$ still satisfies $0<b<\frac{1}{2}$. By small modifications of the proofs of Theorem \ref{th1.2}, one can also obtain scattering for 3d cubic
defocusing $\rm{INLS_{V}}$
\begin{equation}\label{1.15}
\left\{ \begin{aligned}
  i&u_{t}-Hu-|x|^{-b}|u|^{2}u=0,\;\;(t,x) \in {{\bf{R}}\times{\bf{R}}^{3}}, \\
  u&(0, x)=u_{0}(x)\in H^{1}({\bf{R}}^{3}),
 \end{aligned}\right.
 \end{equation}
provided that $u_{0}$ is radial, $0<b<1$ and the confining part of the potential $(x\cdot\nabla V(x))_{+}=\max\{x\cdot\nabla V(x), 0\}$ is small,
precisely,
\begin{equation}\label{1.16}
\|(x\cdot\nabla V(x))_{+}\|_{K}<8\pi.
\end{equation}
In other words, our result extends the result of Dinh \cite{Dinh2} with $0<b<\frac{1}{2}$ into $0<b<1$ in the case of the radial data.
For details of the proof, one can also refer to the one of Theorem B.1 in Hong \cite{H}.
\end{remark}

Finally, we turn to establish the blow-up criterion. To this end, we need introduce another functional associated with the called Virial type identity.
\begin{align}\label{K}
K(u)=\int|\nabla u|^2dx-\frac12\int(x\cdot\nabla V)|u|^2dx-\frac{3+b}4\int|x|^{-b}|u|^4dx.
\end{align}
It follows from Remark \ref{Global} (i) that Theorem 1.1 holds provided that nonnegative $V\in L^{\delta}$ for some $\delta>\frac{3}{2}$. Under some additional assumptions on $V$, that is,
$x\cdot\nabla V\in L^{\delta}$ and $V$ satisfies the following \eqref{Ve1}, we apply the method of Du-Wu-Zhang \cite{DWZ} to obtain a blow-up result, which will be stated as follow.

\begin{theorem}\label{the1}
Suppose that nonnegative $V, \;x\cdot\nabla V\in L^{\delta}$ for some $\delta>\frac{3}{2}$ and $V$ satisfies
\begin{equation}\label{Ve1}
x\cdot\nabla V\leq0,\ \ and\ \ 2V+x\cdot\nabla V\geq0.
 \end{equation}
 We assume that  $0<b<1$ and
\begin{align*}
M(u_{0})^{1-s_{c}}E(u_{0})^{s_{c}}<\mathcal{E}.
\end{align*}
Let $u\in C([0,T_{max}),H^1(\mathbb R^3))$ be the solution to \eqref{1.1} with initial data $u_{0}\in H^{1}({\bf R}^{3})$.
If
\begin{align}\label{bK}
\|u_{0}\|_{L^{2}}^{2(1-s_{c})}\|H^{\frac{1}{2}}u_{0}\|_{L^{2}}^{2s_{c}}>\mathcal{K},
\end{align}
then one of the following two statements holds true:

(i)$T_{max}<\infty$, and
$$\lim_{t\uparrow T_{max}}\|\nabla u(t)\|_{L^2}=\infty.$$

(ii)$T_{max}=\infty$, and there exists a time sequence $\{t_n\}$ such that $t_n\rightarrow\infty$, and
$$\lim_{t_n\uparrow T_{max}}\|\nabla u(t_n)\|_{L^2}=\infty.$$
\end{theorem}

\begin{remark}\label{ree1}
 If $V$ is radial, the condition \eqref{Ve1} implies that $|x|^{-2}\lesssim |V(x)|$ for large $|x|$, which deduces that $V\notin L^{\frac{3}{2}}$.
 So we don't
give the blow up result under the condition \eqref{1.2} in this paper.

\end{remark}

Actually, the proof of Theorem \ref{the1} can be obtained by the following result.
\begin{theorem}\label{the2}
Under the same assumptions as in Theorem \ref{the1}, if there exists $\beta_0<0$ such that there holds
\begin{align}\label{Kbeta}
\sup_{t\in[0,T_{max})}K(u(t))\leq\beta_0<0,
\end{align}
then there exists no global solution $u\in C([0,+\infty),H^1(\mathbb R^3))$ with
\begin{equation}
\sup_{t\in\mathbb R^+}\|\nabla u(t,\cdot)\|_{L^2}<\infty.
\end{equation}
\end{theorem}

This present paper is organized as follows. We fix notations at the end of section 1. In section 2,
We establish Strichartz type estimates, upon
which we obtain linear scattering, local theory, the small data scattering and the perturbation theory.
The variational structure of the ground state of an elliptic problem is given in section 3.
In section 4, we prove a dichotomy proposition of global well-posedness versus blowing up, which
yields the comparability of the total energy and the kinetic energy.
The concentration compactness principle is used in section 5 to give a critical element, which
yields a contradiction through a virial-type estimate in section 6, concluding the proof of Theorem \ref{th1.2}.
In the last section, we use the localized virial identity to give the proofs of Theorem \ref{the1} and Theorem \ref{the2}.

\textbf{Notations:}:

We fix notations used throughout the paper. In what follows, we write $A\lesssim B$ to signify
that there exists a constant $C$ such that $A\leq CB$, while we denote $A\sim B$ when
$A\lesssim B\lesssim A$.

Let $L^{q}=L^{q}({\bf R}^{N})$ be the usual Lebesgue spaces, and $L_{I}^{q}L_{x}^{r}$ or $L^q(I,L^r)$ be the space
of measurable functions from an interval $I\subset {\bf R}$ to $L_{x}^{r}$ whose $L_{I}^{q}L_{x}^{r}$-
norm $
\|\cdot\|_{L_{I}^{q}L_{x}^{r}}$ is finite, where
\begin{align}\label{1.11}
\|u\|_{L_{I}^{q}L_{x}^{r}}=\Big(\displaystyle\int_{I}\|u(t)\|_{L_{x}^{r}}^{q}dt\Big)^{\frac{1}{r}}.
\end{align}
When $I={\bf R}$ or $I=[0, T]$, we may use $L_{t}^{q}L_{x}^{r}$ or $L_{T}^{q}L_{x}^{r}$ instead of
$L_{I}^{q}L_{x}^{r}$, respectively. In particular, when $q=r$, we may simply write them as
$L_{t,x}^{q}$ or $L_{T,x}^{q}$, respectively.

Moreover, the Fourier transform on ${\bf R}^{N}$ is defined by
$\hat{f}(\xi)=(2\pi)^{-\frac{n}{2}}\displaystyle\int_{{\bf R}^{N}} e^{-ix\cdot\xi}f(x)dx$.
For $s\in {\bf R}$, define the inhomogeneous Sobolev space by
$$
H^{s}({\bf R}^{N})=\{f\in S'({\bf R}^{N}):
\displaystyle\int_{{\bf R}^{N}}(1+|\xi|^{2})^{s}|\hat{f}(\xi)|^{2}d\xi<\infty\}
$$
and the homogeneous Sobolev space by
$$\dot{H}^{s}({\bf R}^{N})=\{f\in S'({\bf R}^{N}):
\displaystyle\int_{{\bf R}^{N}}|\xi|^{2s}|\hat{f}(\xi)|^{2}d\xi<\infty\},
$$
where $S'({\bf R}^{N})$ denotes the space of tempered distributions.

Given $p\geq 1$, let $p'$ be the conjugate of $p$, that is $\frac{1}{p}+\frac{1}{p'}=1$.

{\bf Acknowledgement}\ \ The first author is financially supported by the
China National Science Foundation (No.11301564, 11771469), the second author is
financially supported by the
China National Science Foundation ( No. 11101172, 11371158 and
11571131), and the third author is financially supported by the
China National Science Foundation( No. 11371158 and 11771165).



\section{Preliminaries}\label{sec-2}
\setcounter{equation}{0}

We start in this section with recalling the Strichartz estimates and norm equivalence established by Hong \cite{H}.
 We say a pair $(q, r)$ is
Schr\"{o}dinger admissible, or $L^{2}$-admissible, if $2\leq r\leq\infty$ and
$$
\frac{2}{q}+\frac{3}{r}=\frac{3}{2}.
$$
We say that a pair  $(q,r)$ is $\dot{H}^{s}$- admissible and denote it by
 $(q,r)\in\Lambda_{s}$ if $0\leq s<1$, $\frac{6}{3-2s}\leq r\leq 6$ and
$$\frac{2}{q}+\frac{3}{r}=\frac{3}{2}-s$$
Correspondingly, we call the pair $(q',r')$
dual~$\dot{H}^{s}$-admissible, denoted by
$(q',r')\in\Lambda'_{s}$, if
 $(q,r)\in\Lambda_{-s}$, $(\frac{6}{3-2s})^+\leq r\leq 6$ and $(q',r')$ is the conjugate exponent pair of $(q,r).$
 In particular, $(q,r)\in\Lambda_0$ is just a $L^{2}$-admissible pair.

We define the Strichatz norm by
$$
\|u\|_{S(L^{2}, I)}:=\sup_{(q,r):{\rm L^{2}-admissible}}\Vert u\Vert_{L^q(I,L^r)}
$$
and its dual norm by
$$
\|u\|_{S'(L^{2}, I)}:=\inf_{(q,r):{\rm L^{2}-admissible}}\Vert u\Vert_{L^{q'}(I,L^{r'})}
$$

We also define the exotic Strichartz norm by
$$
\|u\|_{S(\dot{H}^{s},I)}:=\sup_{(q,r)\in\Lambda_{s}}\|u\|_{L^q(I;L^r)}
$$
and its dual norm by
$$
\|u\|_{S'(\dot{H}^{-s},I)}:=\inf_{(q,r)\in\Lambda_{-s}}\|u\|_{L^{q'}(I;L^{r'})}
$$

Combining the results obtained by \cite{Keel} and \cite{Fos},
the following Strichartz estimates and Kato inhomogeneous Strichartz estimates on  $I=[0,T]$ are true: If $V$ satisfies \eqref{1.2} and
\eqref{1.3}, then
\begin{align}\label{2.1}
\left\|e^{-itH}f+\int_0^te^{-i(t-s)H}F(\cdot,s)ds\right\|_{S(\dot{H}^{s},I)}\lesssim \|f\|_{\dot{H}^{s}}+\|F\|_{S'(\dot{H}^{-s},I)}.
\end{align}
If the time interval $I$ is not specified, we take $I={\bf R}$, and $S(\dot{H}^{s},I)$ can be abbreviated as $S(\dot{H}^{s})$, similarly
for $S'(\dot{H}^{-s},I)$.

In addition, in order to establish local theory, the two norm equivalent relations between the standard Sobolev norms and the
Sobolev norms associated with $H$ are needed:
If $V$ satisfies \eqref{1.2} and
\eqref{1.3}, then
\begin{align}\label{2.2}
\|H^{\frac{s}{2}}f\|_{L^{r}}\sim\|H_{0}^{\frac{s}{2}}f\|_{L^{r}}
\sim \big\||\nabla|^{s}f\big\|_{L^{r}} \; {\rm and}\; \|(1+ H)^{\frac{s}{2}}f\|_{L^{r}}\sim\|(1+ H_{0})^{\frac{s}{2}}f\|_{L^{r}}
\sim \|\langle\nabla\rangle^{s}f\|_{L^{r}}
\end{align}
where $s\in [0, 2]$ and $1<r<\frac{3}{s}$.

As a simple application of \eqref{2.2}, the following linear scattering result holds.

\begin{lemma}\label{lem2.1} \cite{H}
 (i) If $V$ satisfies \eqref{1.2} and
\eqref{1.3}, then
 for any given $\phi\in L^{2}({\bf R}^{3})$, there exist $\phi^{\pm}$ such that
\begin{align}\label{2.3}
\lim_{t\rightarrow\pm\infty}\|e^{-itH_{0}}\phi-e^{-itH}\phi^{\pm}\|_{L^2({\bf R}^{3})}=0.
\end{align}
(ii) If further assume that $\nabla V\in L^{\frac{3}{2}}$, then
for any given $\phi\in H^{1}({\bf R}^{3})$, there exist $\phi^{\pm}$ such that
\begin{align}\label{2.4}
\lim_{t\rightarrow\pm\infty}\|e^{-itH_{0}}\phi-e^{-itH}\phi^{\pm}\|_{H^1({\bf R}^{3})}=0.
\end{align}
\end{lemma}

We note that the statement and the proof of the following local theory are similar to
those for $({\rm INLS}_{0})$ (see Corollary 1.6 in Guzm$\rm\acute{a}$n \cite{Guz}). The only difference
in the proof is that the norm equivalence is used in several steps.

\begin{lemma}\label{lem2.2} If $V$ satisfies \eqref{1.2} and
\eqref{1.3}, and $u_{0}\in H^{1}({\bf R}^{3})$, then initial value problem
\eqref{1.1} ${\rm INLS_{V}}$ is locally well-posed in $H^{1}({\bf R}^{3})$ and
$$
u\in C([-T, T], H^{1}({\bf R}^{3}))\cap L^{q}([-T, T], H^{1, r}({\bf R}^{3})))
$$
for any $(q, r)$ $L^{2}$-admissible.
\end{lemma}

\vskip0.5cm

Before we show the small data scattering theory, we need three crucial estimates,
which is the key to upgrade the range of $b$ from $0<b<\frac{1}{2}$ to $0<b<1$. In \cite{FG1, FG2}, the
authors divided ${\bf R}^{3}$ into two parts: unit ball $B$ and its complement $B^{C}$ to consider them separately.
However, here we shall rely on the Sobolev inequality (see Theorem $B^*$ in Stein-Weiss \cite{SW})
to get them.

\begin{lemma}\label{lem2.3} Let $u: I\times {\bf R}^{3}\rightarrow {\bf C}$ be a complex function, then the following estimates hold.

(i)
\begin{align}\label{2.5}
\|\nabla (|x|^{-b}|u|^{2}u)\|_{S'(L^{2}, I)}
\lesssim \big\||\nabla|^{s_{c}}u\big\|_{S(L^{2}, I)}\|\nabla u\|_{S(L^{2}, I)}\|u\|_{S(\dot{H}^{s_{c}}, I)}
\lesssim \|\nabla u\|_{S(L^{2}, I)}^{1+s_{c}}\|u\|_{S(L^{2}, I)}^{1-s_{c}}\|u\|_{S(\dot{H}^{s_{c}}, I)},
\end{align}

(ii)
\begin{align}\label{2.6}
\big\||x|^{-b}|u|^{2}u\big\|_{S'(L^{2}, I)}
\lesssim \big\||\nabla|^{s_{c}}u\big\|_{S(L^{2})}\|u\|_{S(L^{2}, I)}
\|u\|_{S(\dot{H}^{s_{c}}, I)}
\lesssim \|\nabla u\|_{S(L^{2}, I)}^{s_{c}}\|u\|_{S(L^{2}, I)}^{2-s_{c}}\|u\|_{S(\dot{H}^{s_{c}}, I)}
\end{align}
and

(iii)
\begin{align}\label{2.7}
\big\||x|^{-b}|u|^{2}u\big\|_{S'(\dot{H}^{-s_{c}}, I)}
\lesssim \big\||\nabla|^{s_{c}} u\big\|_{S(L^{2}, I)}\|u\|_{S(\dot{H}^{s_{c}}, I)}^{2}.
\end{align}
\end{lemma}

\begin{proof}
We first recall the Sobolev inequality.

\begin{lemma}\label{lem2.4}
Let $1< p\leq q'<\infty$, $N\geq 1$, $0<s<N$, and $\alpha, \beta\in{\bf R}$ obey the
conditions
\begin{align*}
\alpha>-\frac{N}{p'},
\end{align*}
\begin{align*}
\beta>-\frac{N}{q'},
\end{align*}
\begin{align*}
\alpha+\beta\leq 0
\end{align*}
and the scaling condition
\begin{align*}
\alpha+\beta-N+s=-\frac{N}{p'}-\frac{N}{q'}.
\end{align*}
Then for any
$u: {\bf R}^{N}\rightarrow {\bf C}$, we have
\begin{align}\label{2.8}
\big\||x|^{\beta}u\big\|_{L^{q'}({\bf R}^{N})}\lesssim_{\alpha,\beta, p, q, s}
\big\||x|^{-\alpha}|\nabla|^{s}u\big\|_{L^{p}({\bf R}^{N})}.
\end{align}
\end{lemma}
Next we give the proof of \eqref{2.5}. Using Leibnitz rule gives
\begin{align}\label{2.9}
\|\nabla (|x|^{-b}|u|^{2}u)\|_{S'(L^{2}, I)}
\lesssim \big\||x|^{-b-1}|u|^{2}u\big\|_{S'(L^{2}, I)}+\big\||x|^{-b}\nabla(|u|^{2}u)\big\|_{S'(L^{2}, I)}.
\end{align}
To control $\big\||x|^{-b-1}|u|^{2}u\big\|_{S'(L^{2})}$, it follows from the definition of
$S'(L^{2})$ and H\"{o}lder inequality that
\begin{equation}\label{2.10}
\begin{split}
\big\||x|^{-b-1}|u|^{2}u\big\|_{S'(L^{2}, I)}
&\lesssim \big\||x|^{-b-1}|u|^{2}u\big\|_{L_{I}^{2}L_{x}^{\frac{6}{5}}}\\
&\lesssim \Big\|\big\||x|^{-s_{c}}u\big\|_{L_{x}^{3}}\big\||x|^{-s_{c}}u\big\|_{L_{x}^{3}}\|u\|_{L_{x}^{6}}\Big\|_{L_{I}^{2}}.
\end{split}
\end{equation}
Using Hardy inequality yields that
\begin{align}\label{2.11}
\big\||x|^{-s_{c}}u\Big\|_{L_{x}^{3}}\lesssim \big\||\nabla|^{s_{c}}u\Big\|_{L_{x}^{3}},
\end{align}
and using \eqref{2.8} with $\beta=-s_{c}$, $q'=3$, $\alpha=0$, $s=1$ and $p=\frac{6}{3-b}$ gives
\begin{align}\label{2.12}
\big\||x|^{-s_{c}}u\big\|_{L_{x}^{3}}\lesssim \big\||\nabla|u\big\|_{L_{x}^{\frac{6}{3-b}}}.
\end{align}
Substituting \eqref{2.11} and \eqref{2.12} in the \eqref{2.10}, using H\"{o}lder inequality in
the time variable $t$ and noting that $(4, 3), (\frac{4}{b}, \frac{6}{3-b})\in \Lambda_{0}$
and $(\frac{4}{1-b}, 6)\in \Lambda_{s_{c}}$, we have
\begin{equation}\label{2.13}
\begin{split}
\big\||x|^{-b-1}|u|^{2}u\big\|_{S'(L^{2}, I)}
&\lesssim \Big\|\big\||\nabla|^{s_{c}}u\big\|_{L_{x}^{3}}\big\||\nabla|u\big\|_{L_{x}^{\frac{6}{3-b}}}\|u\|_{L_{x}^{6}}\Big\|_{L_{I}^{2}}\\
&\lesssim \big\||\nabla|^{s_{c}}u\big\|_{L_{I}^{4}L_{x}^{3}}\big\||\nabla|u\big\|_{L_{I}^{\frac{4}{b}}L_{x}^{\frac{6}{3-b}}}
\|u\|_{L_{I}^{\frac{4}{1-b}}L_{x}^{6}}\\
&\lesssim \big\||\nabla|^{s_{c}}u\big\|_{S(L^{2}, I)}\big\||\nabla|u\big\|_{S(L^{2}, I)}
\|u\|_{S(\dot{H}^{s_{c}}, I)}\\
&\lesssim \|\nabla u\|_{S(L^{2}, I)}^{1+s_{c}}\|u\|_{S(L^{2}, I)}^{1-s_{c}}\|u\|_{S(\dot{H}^{s_{c}}, I)},
\end{split}
\end{equation}
where in the last step we have used the interpolation.

To control $\big\||x|^{-b}\nabla(|u|^{2}u)\big\|_{S'(L^{2})}$, we apply Leibnitz rule and H\"{o}lder inequality to get
\begin{equation}\label{2.14}
\begin{split}
\big\||x|^{-b}\nabla(|u|^{2}u)\big\|_{S'(L^{2}, I)}
&\lesssim \big\||x|^{-b}u^{*}u^{*}\nabla u^{*}\big\|_{L_{I}^{2}L_{x}^{\frac{6}{5}}}\\
&\lesssim \Big\|\big\||x|^{-b}u^{*}\big\|_{L_{x}^{\frac{6}{1+b}}}\|u^{*}\|_{L_{x}^{6}}\|\nabla u^{*}\|_{L_{x}^{\frac{6}{3-b}}}\Big\|_{L_{I}^{2}}\\
&\lesssim \Big\|\big\||x|^{-b}u\big\|_{L_{x}^{\frac{6}{1+b}}}\|u\|_{L_{x}^{6}}\|\nabla u\|_{L_{x}^{\frac{6}{3-b}}}\Big\|_{L_{I}^{2}}
\end{split}
\end{equation}
where $u^{*}$ is either $u$ or $\bar{u}$.
By \eqref{2.8} with $\beta=-b$, $q'=\frac{6}{1+b}$, $\alpha=0$, $s=s_{c}=\frac{1+b}{2}$ and $p=3$, we have
\begin{align}\label{2.15}
\big\||x|^{-b}u\Big\|_{L_{x}^{\frac{6}{1+b}}}\lesssim \big\||\nabla|^{s_{c}}u\Big\|_{L_{x}^{3}}.
\end{align}
Substituting \eqref{2.15} in the \eqref{2.14}, using H\"{o}lder inequality in
the time variable $t$ and noting that $(4, 3), (\frac{4}{b}, \frac{6}{3-b})\in \Lambda_{0}$
and $(\frac{4}{1-b}, 6)\in \Lambda_{s_{c}}$, we have
\begin{equation}\label{2.16}
\begin{split}
\big\||x|^{-b}\nabla(|u|^{2}u)\big\|_{S'(L^{2}, I)}
&\lesssim\Big\|\big\||\nabla|^{s_{c}}u\big\|_{L_{x}^{3}}\big\||\nabla|u\big\|_{L_{x}^{\frac{6}{3-b}}}\|u\|_{L_{x}^{6}}\Big\|_{L_{I}^{2}}\\
&\lesssim \big\||\nabla|^{s_{c}}u\big\|_{L_{I}^{4}L_{x}^{3}}\big\||\nabla|u\big\|_{L_{I}^{\frac{4}{b}}L_{x}^{\frac{6}{3-b}}}
\|u\|_{L_{I}^{\frac{4}{1-b}}L_{x}^{6}}\\
&\lesssim \big\||\nabla|^{s_{c}}u\big\|_{S(L^{2}, I)}\big\||\nabla|u\big\|_{S(L^{2}, I)}
\|u\|_{S(\dot{H}^{s_{c}}, I)}\\
&\lesssim \|\nabla u\|_{S(L^{2}, I)}^{1+s_{c}}\|u\|_{S(L^{2}, I)}^{1-s_{c}}\|u\|_{S(\dot{H}^{s_{c}}, I)},
\end{split}
\end{equation}
where in the last step we have used the interpolation.

Putting \eqref{2.9}, \eqref{2.13} and \eqref{2.16} together, we complete the proof of \eqref{2.5}.

From the process for \eqref{2.14}-\eqref{2.16}, we easily obtain that
\begin{equation}\label{2.17}
\begin{split}
\big\||x|^{-b}|u|^{2}u\big\|_{S'(L^{2}, I)}
&\lesssim \big\||\nabla|^{s_{c}}u\big\|_{S(L^{2})}\|u\|_{S(L^{2}, i)}
\|u\|_{S(\dot{H}^{s_{c}}, I)}\\
&\lesssim \|\nabla u\|_{S(L^{2}, I)}^{s_{c}}\|u\|_{S(L^{2}, i)}^{2-s_{c}}\|u\|_{S(\dot{H}^{s_{c}}, i)},
\end{split}
\end{equation}

Finally, we turn to the estimate of \eqref{2.7}.
we apply Leibnitz rule and H\"{o}lder inequality to get
\begin{equation}\label{2.18}
\begin{split}
\big\||x|^{-b}|u|^{2}u\big\|_{S'(\dot{H}^{-s_{c}}, I)}
&\lesssim \big\||x|^{-b}u^{*}u^{*} u^{*}\big\|_{L_{I}^{\frac{4}{1-b}}L_{x}^{\frac{6}{5}}}\\
&\lesssim \Big\|\big\||x|^{-b}u^{*}\big\|_{L_{x}^{\frac{6}{2+b}}}\|u^{*}\|_{L_{x}^{6}}\| u^{*}\|_{L_{x}^{\frac{6}{2-b}}}\Big\|_{L_{I}^{2}}\\
&\lesssim \Big\|\big\||x|^{-b}u\big\|_{L_{x}^{\frac{6}{2+b}}}\|u\|_{L_{x}^{6}}\| u\|_{L_{x}^{\frac{6}{2-b}}}\Big\|_{L_{I}^{2}}.
\end{split}
\end{equation}
By \eqref{2.8} with $\beta=-b$, $q'=\frac{6}{2+b}$, $\alpha=0$, $s=s_{c}=\frac{1+b}{2}$ and $p=2$, we have
\begin{align}\label{adding}
\big\||x|^{-b}u\big\|_{L_{x}^{\frac{6}{2+b}}}\lesssim \big\||\nabla|^{s_{c}}u\big\|_{L_{x}^{2}}.
\end{align}
Substituting \eqref{adding} in the \eqref{2.18}, using H\"{o}lder inequality in
the time variable $t$ and noting that $(\infty, 2)\in \Lambda_{0}$
and $(\frac{4}{1-b}, 6), (\infty, \frac{6}{2-b})\in \Lambda_{s_{c}}$, we have
\begin{equation}\label{2.19}
\begin{split}
\big\||x|^{-b}|u|^{2}u\big\|_{S'(\dot{H}^{-s_{c}}, I)}
&\lesssim \Big\|\big\||\nabla|^{s_{c}}u\big\|_{L_{x}^{2}}\|u\|_{L_{x}^{6}}\|u\|_{L_{x}^{\frac{6}{2-b}}}\Big\|_{L_{t}^{2}}\\
&\lesssim \big\||\nabla|^{s_{c}}u\big\|_{L_{I}^{\infty}L_{x}^{2}}\|u\|_{L_{I}^{\frac{4}{1-b}}L_{x}^{6}}
\|u\|_{L_{I}^{\infty}L_{x}^{\frac{6}{2-b}}}\\
&\lesssim \big\||\nabla|^{s_{c}}u\big\|_{S(L^{2}, I)}
\|u\|_{S(\dot{H}^{s_{c}}, I)}^{2}.
\end{split}
\end{equation}
\end{proof}

Once obtaining nonlinear estimates \eqref{2.5}, \eqref{2.6} and \eqref{2.7}, we can follow the standard procedure (e.g.
see \cite{HR, H, FG1}) and use Kato inhomogeneous Strichartz estimates \eqref{2.1} and the normal equivalence \eqref{2.2}
to get the following small data global wellposedness Proposition \ref{pro2.5}, scattering criterion Proposition \ref{pro2.6}
and stability result Lemma \ref{lem2.9}, so we omit their proofs here. It is easy to see that combining Proposition \ref{pro2.5} with
Proposition \ref{pro2.6} yields a small data scattering result.

\begin{proposition}\label{pro2.5}
Let $V$ satisfy \eqref{1.2} and
\eqref{1.3}.
Assume $u_{0}\in H^1({\bf R}^{3})$ and $\|u_{0}\|_{H^{1}}\leq A$. Then
there exists  $\delta_{sd}>0$ depending on $A$ such that if
 $\|e^{-itH}u_{0}\|_{S(\dot{H}^{s_{c}})}\leq \delta_{sd}, $ then there exists a unique global solution
 $u$ of \eqref{1.1} with initial data $u_0$. Furthermore,
\begin{align}\label{2.20}
 \|u\|_{S(\dot{H}^{s_{c}})}\leq
2\|e^{-itH}u_{0}\|_{S(\dot{H}^{s_{c}})},\ \ \ \|\langle \nabla\rangle u\|_{S(L^{2})}\leq 2c\|u_0\|_{H^1}.
\end{align}
\end{proposition}

\begin{proposition}\label{pro2.6}
Let $V$ satisfy \eqref{1.2} and
\eqref{1.3},
and let $u(t)\in C({\bf R},H^1({\bf R}^{3}))$ be a radial solution of \eqref{1.1} such that $\sup_{t\in {\bf R}}\|u(t)\|_{H^{1}}<+\infty$,
If $\|u\|_{S(\dot{H}^{s_{c}})}<+\infty$,
then $u(t)$ scatters in $H^1({\bf R}^{3})$.
That is , there exists $\phi^{\pm}\in H^1({\bf R}^{3})$ such that
$$
\lim_{t\rightarrow\pm\infty}\|u(t)-e^{-itH}\phi^{\pm}\|_{H^1({\bf R}^{3})}=0.
$$
\end{proposition}

\begin{lemma}\label{lem2.9}
Let $V$ satisfy \eqref{1.2} and
\eqref{1.3}. For any given $M>0, M'>0$ and $L>0$, there exists $\epsilon_{0}=\epsilon_{0}(M, M', L)$
and $c(M, M', L)>0$ such that for any $0<\epsilon<\epsilon_{0}$,
if $\tilde{u}_{0}, u_{0}\in H^{1}({\bf R}^{3})$ are radial functions and
$\tilde{u}=\tilde{u}(t, x)\in C(I, H^1({\bf R}^{3}))$ is a radial solution to
$$
 i\tilde{u}_{t}+H \tilde{u}-|x|^{-b}|\tilde{u}|^{2}\tilde{u}=e,
 $$
 with $\tilde{u}(0, x)=\tilde{u}_{0}\in H^{1}({\bf R}^{3})$ satisfying
 \begin{align}\label{2.72}
 \sup_{t\in I}\|\tilde{u}(t)\|_{H_{x}^{1}}\leq M\; \text{and}\;   \|\tilde{u}\|_{S(\dot{H}^{s_{c}}, I)}\leq L
 \end{align}
\begin{align}\label{2.73}
\|e^{-itH}(u_{0}-\tilde{u}_{0})\|_{S(\dot{H}^{s_{c}}, I)}\leq\epsilon\;\text{and}\;
\|u_{0}-\tilde{u}_{0}\|_{H^{1}}\leq M'
\end{align}
and
 \begin{align}\label{2.74}
 \|\langle\nabla\rangle e\|_{S'(L^{2}, I)}+\|e\|_{S'(\dot{H}^{-s_{c}}, I)}\leq \epsilon,
 \end{align}
then
there is a unique solution $u\in C(I, H^1({\bf R}^{3}))$ to \eqref{1.1} $({\rm INS}_{V})$
with $u(0, x)=u_{0}$ satisfying
\begin{align}\label{2.75}
 \|u-\tilde{u}\|_{S(\dot{H}^{s_{c}}, I)}\leq c(M, M', L)\epsilon
\end{align}
 and
\begin{align}\label{2.76}
  \|u\|_{S(\dot{H}^{s_{c}}, I)}+\|\langle\nabla\rangle u\|_{S(L^{2}, I)}\leq c(M, M', L).
\end{align}
 \end{lemma}

\vskip0.5cm

\section{Sharp Gagliardo-Nirenberg inequality}\label{sec-5}
\setcounter{equation}{0}

In this section, we consider a maximizer or maximizing sequence of the nonlinear functional
\begin{align}\label{3.1}
J_{V}(u)=\frac{\big\||x|^{-b} |u|^{4}\big\|_{L^{1}}}{\|u\|_{L^2}^{1-b}\Big(\|\nabla u\|_{L^2}^{2}
+\displaystyle\int_{{\bf R}^{3}}V|u|^{2}dx\Big)^{\frac{3+b}{2}}}
\end{align}
in the two cases $V_{-}=0$ and $V_{-}\neq 0$.
To make this precise, we define
$$
C_{GN}^{rad}=\sup\{J_{V}(u): u\in H^{1}({\bf{R}}^{3}), u\;\text{is radial and nonzero}\}
$$
and
$$
C_{GN}=\sup\{J_{V}(u): u\in H^{1}({\bf{R}}^{3}), u\;\text{ is nonzero}\}.
$$

It's known from \cite{Far} that for $V=0$, $J_{0}(u)$ attains its maximum $J_{0}$  at $u=Q(x)\geq 0$, which solves the equation \eqref{1.7} with $p=3$, and
\begin{align}\label{3.2}
J_{0}=J_{0}(Q)=\frac{\big\||x|^{-b}|Q|^{4}\big\|_{L^{1}}}{\|Q\|_{L^2}^{1-b}\|\nabla Q\|_{L^2}^{3+b}},
\end{align}
which together with the identities
\begin{align}\label{3.3}
\|\nabla Q\|_{L^2}^2=\frac{3+b}{1-b}\|Q\|_{L^{2}}^{2},\ \
\||x|^{-b}|Q|^{4}\|_{L^{1}}=\frac{4}{3+b}\|\nabla Q\|_{L^{2}}^{2},\ \
E_0(Q)=\frac{s_{c}}{3+b}\|\nabla Q\|_{L^{2}}^{2},
\end{align}
implies that the best constant of the Gagliardo-Nirenberg inequality
\begin{align}\label{3.4}
\big\||x|^{-b}|u|^{4}\big\|_{L^{1}}\leq C^0_{GN} \|u\|_{L^2}^{1-b}\|\nabla
u\|_{L^2}^{3+b}
\end{align}
is
\begin{align}\label{3.5}
 C^0_{GN}=J_{0}=\frac4{(3+b)\|Q\|_{L^2}^{2(1-s_{c})}\|\nabla
Q\|_{L^2}^{2s_{c}} }.
\end{align}

\begin{lemma}\label{lem3.1}
If $V\geq 0$, then $\{Q_{n}(x)\}_{n=1}^{\infty}$ is a maximizing
sequence for $J_{V}(u)$, where $Q_{n}(x)=\frac{1}{n^{\frac{2-b}{2}}}Q(\frac x n)$.
\end{lemma}
\begin{proof}
it follows from \eqref{3.2}, \eqref{3.4} and \eqref{3.5} that
\begin{align}\label{3.6}
J_{0}(Q)\geq J_{0}(u).
\end{align}
On one hand,
\begin{align}\label{3.7}
\lim_{n\rightarrow\infty}J_{V}(Q_{n})=J_{0}(Q),
\end{align}
which follows from
\begin{align*}
\displaystyle\int_{{\bf R}^{N}}n^{2}V(nx)Q(x)^{2}dx\lesssim \|V\|_{L^{\frac{3}{2}}}\|Q\|_{L^{6}}^{2}
\lesssim \|V\|_{L^{\frac{3}{2}}}\|\nabla Q\|_{L^2}^{2}
\end{align*}
and
\begin{align}\label{3.8}
J_{V}(Q_{n})=\frac{\big\||x|^{-b} |Q|^{4}\big\|_{L^{1}}}{\|Q\|_{L^2}^{1-b}\Big(\|\nabla Q\|_{L^2}^{2}
+\displaystyle\int_{{\bf R}^{3}}n^{2}V(nx)|Q|^{2}dx\Big)^{\frac{3+b}{2}}}.
\end{align}
On the other hand, for $V\geq 0$, it is easy to see that for any $u\in H^1$,
\begin{align}\label{3.9}
J_{0}(u)\geq J_{V}(u)
\end{align}
Putting \eqref{3.6}, \eqref{3.7} and \eqref{3.9} together yields that for any $u\in H^1$
\begin{align}\label{3.10}
\lim_{n\rightarrow\infty}J_{V}(Q_{n})\geq J_{V}(u)
\end{align}
Thus, we get our desired result.
\end{proof} \vskip0.5cm

\begin{remark}\label{rem3.2}
(i) It follows from lemma 3.1 that there hold that $J_0(Q)=\lim_{n\rightarrow\infty}J_{V}(Q_{n})\geq J_V(u)$
for any $u$, which implies that with the same Gagliardo-Nirenberg constant  ($C_{GN}=C_{GN}^{0}$ \eqref{3.5}), there holds the following sharp inequality: 
\begin{align}\label{3.11}
\big\||x|^{-b}|u|^{4}\big\|_{L^{1}}\leq C_{GN}^{0} \|u\|_{L^2}^{1-b}\|H^{\frac{1}{2}}
u\|_{L^2}^{3+b}.
\end{align}
In the case when $V$ is nonegative and not zero a.e. on ${\bf R}^3$, the constant $C_{GN}=C_{GN}^{0}$ can never be attained. In fact, if not, then there exists some $\tilde Q\in H^1({\bf R}^3)$
such that $C_{GN}=J_V(\tilde Q)<J_0(\tilde Q)\leq J_0( Q)=C_{GN}^{0}=C_{GN}$, which is a contradiction.

(ii) Since $\{Q_{n}(x)\}_{n=1}^{\infty}$  is a radial sequence, the arguments in Lemma \ref{lem3.1} and (i) still work for radial functions. So we can find that $C_{GN}^{rad}=C_{GN}^{0}=C_{GN}$ and it is never attained,
which is different from the case that $V$ is an inverse-square potential $\frac{a}{|x|^{2}}$(see \cite{KMVZ}), where $a>\frac{1}{4}$. In \cite{KMVZ}, the authors showed that $C_{GN}^{rad}$ can be attained but
$C_{GN}$ cannot be and $C_{GN}^{rad}<C_{GN}$ provided that $a>0$.
 \end{remark}
 \medskip
When $V_-\neq0$, we also have that $C_{GN}^{rad}=C_{GN}$ which can be attained further.
More precisely, we have the following.

\begin{lemma}\label{lem3.5}
If $V$ is radially symmetric and $V_-\neq0$, then the sharp constant $C_{GN}^{rad}=C_{GN}$ can be attained by a radially symmetric  function $\mathcal{Q}$, that is,
there exists a maximizer $\mathcal{Q}$ for $J_{V}(u)$, where $\mathcal{Q}$ solves the elliptic equation
\begin{align}\label{3.33}
(-\Delta +V)\mathcal{Q}+w_{\mathcal{Q}}^{2}\mathcal{Q}-|x|^{-b}|\mathcal{Q}|^{2}\mathcal{Q}=0, \;
w_{\mathcal{Q}}=\frac{\sqrt{1-b}}{\sqrt{3+b}}\frac{\|H^{\frac{1}{2}}\mathcal{Q}\|_{L^{2}}}{\|\mathcal{Q}\|_{L^{2}}}.
\end{align}
Moreover, $\mathcal{Q}$ satisfies the Pohozhaev identities,
\begin{align}\label{3.34}
\|H^{\frac{1}{2}}\mathcal{Q}\|_{L^{2}}=\frac{3+b}{1-b}w_{\mathcal{Q}}^{2}\|\mathcal{Q}\|_{L^{2}},\;
\big\||x|^{-b}|\mathcal{Q}|^{4}\big\|_{L^{1}}=\frac{4}{1-b}w_{\mathcal{Q}}^{2}\|\mathcal{Q}\|_{L^{2}}^{2}.
\end{align}
\end{lemma}

 \begin{proof}
 Set $$I(u)=\int_{{\bf R}^{3}}|x|^{-b}|u(x)|^4dx.$$
Let $\{u_{n}\}_{n=1}^{\infty}\subset H^1({\bf R}^3)$ be a maximizing sequence associated to $J_{V}(u)$.
By Schwarz symmetrization, we can assume that $\{u_{n}\}_{n=1}^{\infty}$ is radial and radially non-increasing for all $n$.
For each $n$, we choose $\alpha_{n}$, $r_{n}>0$ such that
\begin{align*}
\|\alpha_{n}u(\frac{\cdot}{r_{n}})\|_{L^{2}}^{2}=\alpha_{n}^{2}r_{n}^{3}\|u_{n}\|_{L^{2}}^{2}=1
\end{align*}
and
\begin{align*}
\|H_{r_{n}}^{\frac{1}{2}}\alpha_{n}u(\frac{\cdot}{r_{n}})\|_{L^{2}}^{2}=\alpha_{n}^{2}r_{n}\|H^{\frac{1}{2}}u_{n}\|_{L^{2}}^{2}=1
\end{align*}
where $H_{r}=-\Delta+\frac{1}{r^{2}}V(\frac{\cdot}{r})$. Since $J_{V}(\alpha u)=J_{V}(u)$, replacing $\{u_{n}\}_{n=1}^{\infty}$
by $\{\alpha_{n}u_{n}\}_{n=1}^{\infty}$, we may assume that $\|u(\frac{\cdot}{r_{n}})\|_{L^{2}}^{2}=1$ and
$\|H_{r_{n}}^{\frac{1}{2}}u(\frac{\cdot}{r_{n}})\|_{L^{2}}^{2}=1$. Set $\tilde{u}_{n}=u_{n}(\frac{\cdot}{r_{n}})$. Then
$\{\tilde{u}_{n}\}_{n=1}^{\infty}$ is a bounded sequence in $H^{1}$, since by the norm equivalence,
\begin{align*}
\|\tilde{u}_{n}\|_{L^{2}}^{2}=1,\; \|\nabla \tilde{u}_{n}\|_{L^{2}}^{2}\sim \|H^{\frac{1}{2}}_{r_{n}}\tilde{u}_{n}\|_{L^{2}}^{2}=1.
\end{align*}
Therefore, there exists some $\tilde u\in H^1({\bf R}^3)$ such that, up to a subsequence, $\tilde u_n\rightharpoonup\tilde u$ weakly in
$H^1({\bf R}^3)$.
Furthermore, $\tilde u$ is nonnegative, spherically symmetric, radially non-increasing, and with some $r_0\in(0,+\infty)$:
$$\|\tilde u\|_{L^2}\leq1,\ \ \|H^{\frac12}_{r_0}\tilde u\|_{L^2}=\left(\int_{{\bf R}^3}|\nabla\tilde u|^2+\frac1{r_0^2}V(\frac{x}{r_0})|\tilde u|^2dx\right)^{\frac12}\leq1.$$
Indeed, if we suppose that $r_n\rightarrow0$ or $r_n\rightarrow+\infty$, then by the "free" Gagliardo-Nirenberg inequality and the assumption,
\begin{align}\label{re1}
\frac{\big\||x|^{-b}|Q|^{4}\big\|_{L^{1}}}
{\|Q\|_{L^{2}}^{1-b}\|\nabla Q\|_{L^{2}}^{3+b}}\geq
\lim_{n\rightarrow\infty}\frac{\big\||x|^{-b}|\tilde{u}_{n}|^{4}\big\|_{L^{1}}}
{\|\tilde{u}_{n}\|_{L^{2}}^{1-b}\|\nabla\tilde{u}_{n}\|_{L^{2}}^{3+b}}
=\lim_{n\rightarrow\infty}\frac{\big\||x|^{-b}|\tilde{u}_{n}|^{4}\big\|_{L^{1}}}
{\|\tilde{u}_{n}\|_{L^{2}}^{1-b}\|H_{r_{n}}^{\frac{1}{2}}\tilde{u}_{n}\|_{L^{2}}^{3+b}}= C_{GN}.
\end{align}
with $Q$ is the ground state of the free equation.
On the other hand, since $V_-\neq0$, then there exists some $x^*\in\mathbb R^3$ and a small $\epsilon>0$ such that
$$\int_{{\bf R}^3}V(x)Q^2\left(\frac{x-x^*}\epsilon\right)dx<0.$$ 
Hence,
$$C_{GN}\geq J_V\left(Q\left(\frac{x-x^*}\epsilon\right)\right)>\frac{\big\||x|^{-b}|Q(\frac{x-x^*}\epsilon)|^{4}\big\|_{L^{1}}}
{\|Q(\frac{x-x^*}\epsilon)\|_{L^{2}}^{1-b}\|\nabla Q(\frac{x-x^*}\epsilon)\|_{L^{2}}^{3+b}}=\frac{\big\||x|^{-b}|Q|^{4}\big\|_{L^{1}}}
{\|Q\|_{L^{2}}^{1-b}\|\nabla Q\|_{L^{2}}^{3+b}},$$
contradicting \eqref{re1}.

In this stage, we set $\psi(x)=\tilde u(r_0x)$ and obtain that
\begin{align*}
C_{GN}&=\lim_{n\rightarrow\infty}J_{V}(u_{n})
=\lim_{n\rightarrow\infty}\frac{\big\||x|^{-b}|u_{n}|^{4}\big\|_{L^{1}}}
{\|u_{n}\|_{L^{2}}^{1-b}\|H^{\frac{1}{2}}u_{n}\|_{L^{2}}^{3+b}}
=\lim_{n\rightarrow\infty}\frac{\big\||x|^{-b}|\tilde{u}_{n}|^{4}\big\|_{L^{1}}}
{\|\tilde{u}_{n}\|_{L^{2}}^{1-b}\|H_{r_{n}}^{\frac{1}{2}}\tilde{u}_{n}\|_{L^{2}}^{3+b}}\nonumber\\
&\leq\frac{\big\||x|^{-b}|\tilde{u}|^{4}\big\|_{L^{1}}}
{\|\tilde{u}\|_{L^{2}}^{1-b}\|H_{r_{0}}^{\frac{1}{2}}\tilde{u}\|_{L^{2}}^{3+b}}
=\frac{\big\||x|^{-b}|\psi|^{4}\big\|_{L^{1}}}
{\|\psi\|_{L^{2}}^{1-b}\|H
^{\frac{1}{2}}\psi\|_{L^{2}}^{3+b}}\leq C_{GN}.
\end{align*}
Therefore, we actually obtain that $\tilde u_n\rightarrow\tilde u$ and $r_n\rightarrow r_0$ which give then $u_n\rightarrow\psi$
in $H^1$, attaining $C_{GN}$.

Now that $\psi$ is a maximizer of $J_{V}(u)$. Then, it solves the
Euler-Lagrangle equation equivalently,
\begin{align*}
\langle H\psi+\frac{1-b}{3+b}\frac{\|H^{\frac{1}{2}}\psi\|_{L^{2}}^{2}}{\|\psi\|_{L^{2}}^{2}}\psi
-\frac{4}{3+b}\frac{\|H^{\frac{1}{2}}\psi\|_{L^{2}}^{2}}{\||x|^{-b}|\psi|^{4}\|_{L^{1}}}|x|^{-b}|\psi|^{2}\psi,
V \rangle=0
\end{align*}
for all $v\in H^{1}$. We set
\begin{align*}
\mathcal{Q}\doteq \frac{2}{\sqrt{3+b}}\frac{\|H^{\frac{1}{2}}\psi\|_{L^{2}}}{\||x|^{-b}|\psi|^{4}\|_{L^{1}}^{\frac{1}{2}}}\psi.
\end{align*}
Then $\mathcal{Q}$ is a weak solution to the ground state equation \eqref{3.33}.

Let's turn to the proof of \eqref{3.34}. Formally, multiplying \eqref{3.33} by $\mathcal{Q}$ and $x\cdot\nabla\mathcal{Q}$, separately,  integrating in $x$
and applying integration by parts, we get
\begin{align}\label{3.46}
\|H^{\frac{1}{2}}\mathcal{Q}\|_{L^{2}}^{2}+w_{\mathcal{Q}}^{2}\|\mathcal{Q}\|_{L^{2}}^{2}-\big\||x|^{-b}|\mathcal{Q}|^{4}\big\|_{L^{1}}=0
\end{align}
and
\begin{align}\label{3.47}
\|H^{\frac{1}{2}}\mathcal{Q}\|_{L^{2}}^{2}+3w_{\mathcal{Q}}^{2}\|\mathcal{Q}\|_{L^{2}}^{2}-\frac{3-b}{2}\big\||x|^{-b}|\mathcal{Q}|^{4}\big\|_{L^{1}}
+\displaystyle\int_{{\bf R}^{3}}(2V+x\cdot\nabla V)|\mathcal{Q}|^{2}dx=0
\end{align}
The rigorous proof relies on the standard approximating method.
Solving the simultaneous equations \eqref{3.46} and \eqref{3.47} in $\|H^{\frac{1}{2}}\mathcal{Q}\|_{L^{2}}^{2}$ and
$\big\||x|^{-b}|\mathcal{Q}|^{4}\big\|_{L^{1}}$ gives
\begin{align}\label{3.48}
\|H^{\frac{1}{2}}\mathcal{Q}\|_{L^{2}}^{2}=\frac{3+b}{1-b}w_{\mathcal{Q}}^{2}\|\mathcal{Q}\|_{L^{2}}^{2}
+\frac{2}{1-b}\displaystyle\int_{{\bf R}^{3}}(2V+x\cdot\nabla V)|\mathcal{Q}|^{2}dx,
\end{align}
and
\begin{align}\label{3.48}
\||x|^{-b}|\mathcal{Q}|^{4}\|_{L^{1}}=\frac{4}{1-b}w_{\mathcal{Q}}^{2}\|\mathcal{Q}\|_{L^{2}}^{2}
+\frac{2}{1-b}\displaystyle\int_{{\bf R}^{3}}(2V+x\cdot\nabla V)|\mathcal{Q}|^{2}dx.
\end{align}
Substituting $$w_{\mathcal{Q}}=\frac{\sqrt{1-b}}{\sqrt{3+b}}\frac{\|H^{\frac{1}{2}}\mathcal{Q}\|_{L^{2}}}{\|\mathcal{Q}\|_{L^{2}}}$$
 in \eqref{3.46} and \eqref{3.47}
yields that
\begin{align}\label{3.49}
\frac{2}{1-b}\displaystyle\int_{{\bf R}^{3}}(2V+x\cdot\nabla V)|\mathcal{Q}|^{2}dx=0.
\end{align}
Then \eqref{3.48} and \eqref{3.49} imply that \eqref{3.34} is true.
\end{proof}

\section{Criteria for global well-posedness}
\setcounter{equation}{0}
In this section we first apply the results in the previous section to give the proof of Theorem \ref{th1.1}, and then establish
some estimates required in the proof of Theorem \ref{th1.2}.
For one's convenience, we restate Theorem \ref{th1.1} as follows.

\begin{theorem}\label{th4.1}
Suppose that $V$ is radially symmetric and satisfies \eqref{1.2} and \eqref{1.3}, and $0<b<1$. We assume that
\begin{align}\label{4.1}
M(u_{0})^{1-s_{c}}E(u_{0})^{s_{c}}<\mathcal{E}.
\end{align}
Let $u(t)$ be the solution to \eqref{1.1} with initial data $u_{0}\in H^{1}({\bf R}^{3})$.

(i) If
\begin{align}\label{4.2}
\|u_{0}\|_{L^{2}}^{2(1-s_{c})}\|H^{\frac{1}{2}}u_{0}\|_{L^{2}}^{2s_{c}}<\mathcal{K},
\end{align}
then $u(t)$ exists globally in time, and
\begin{align}\label{4.3}
\|u_{0}\|_{L^{2}}^{2(1-s_{c})}\|H^{\frac{1}{2}}u(t)\|_{L^{2}}^{2s_{c}}<\mathcal{K}, \forall t\in{\bf R}.
\end{align}

(ii) If
\begin{align}\label{4.4}
\|u_{0}\|_{L^{2}}^{2(1-s_{c})}\|H^{\frac{1}{2}}u_{0}\|_{L^{2}}^{2s_{c}}>\mathcal{K},
\end{align}
then
\begin{align}\label{4.5}
\|u_{0}\|_{L^{2}}^{2(1-s_{c})}\|H^{\frac{1}{2}}u(t)\|_{L^{2}}^{2s_{c}}>\mathcal{K}
\end{align}
during the maximal existence time.
\end{theorem}
\vskip 0.5cm

\begin{proof}
If $V\geq 0$, it follows from \eqref{3.3} and \eqref{3.5} that
\begin{align}\label{4.6}
\mathcal{E}=M(Q)^{1-s_{c}}E_{0}(Q)^{s_{c}}
=\Big(\frac{s_{c}}{3+b}\Big)^{s_{c}}\|Q\|_{L^{2}}^{2(1-s_{c})}\|\nabla Q\|_{L^{2}}^{2s_{c}}
=\Big(\frac{s_{c}}{3+b}\Big)^{s_{c}}\mathcal{K}
\end{align}
and
\begin{align}\label{4.7}
 C_{GN}=\frac4{(3+b)\|Q\|_{L^2}^{2(1-s_{c})}\|\nabla
Q\|_{L^2}^{2s_{c}} }=\frac4{(3+b)\mathcal{K}}.
\end{align}
If $V\leq 0$, using Pohozhaev identities \eqref{3.34}, we have
\begin{align}\label{4.8}
E(\mathcal{Q})=\frac{1}{2}\|H^{\frac{1}{2}}\mathcal{Q}\|_{L^{2}}^{2}-\frac{1}{4}\big\||x|^{-b}|\mathcal{Q}|^{4}\big\|_{L^{1}}
=\frac{s_{c}}{3+b}\|H^{\frac{1}{2}}\mathcal{Q}\|_{L^{2}}^{2},
\end{align}
which implies that
\begin{align}\label{4.9}
\mathcal{E}=M(\mathcal{Q})^{1-s_{c}}E(\mathcal{Q})^{s_{c}}
=\Big(\frac{s_{c}}{3+b}\Big)^{s_{c}}\|\mathcal{Q}\|_{L^{2}}^{2(1-s_{c})}\|H^{\frac{1}{2}}\mathcal{Q}\|_{L^{2}}^{2s_{c}}
=\Big(\frac{s_{c}}{3+b}\Big)^{s_{c}}\mathcal{K}.
\end{align}
Using Pohozhaev identities \eqref{3.34} again, we have
\begin{align}\label{4.10}
 C_{GN}=\frac{\big\||x|^{-b} |\mathcal{Q}|^{4}\big\|_{L^{1}}}{\|\mathcal{Q}\|_{L^2}^{1-b}\|H^{\frac{1}{2}} \mathcal{Q}\|_{L^2}^{3+b}}
 =\frac4{(3+b)\|\mathcal{Q}\|_{L^2}^{2(1-s_{c})}\|H^{\frac{1}{2}} \mathcal{Q}\|_{L^2}^{2s_{c}} }=\frac4{(3+b)\mathcal{K}}.
\end{align}
Hence, using the Gagliardo-Nirenberg inequality( see Remark \ref{rem3.2} and Lemma \ref{lem3.5}) and the energy conservation law \eqref{1.9} yields that
\begin{align}\label{4.11}
\begin{split}
\mathcal{E}&>M(u_{0})^{1-s_{c}}E(u_{0})^{s_{c}}=M(u_{0})^{1-s_{c}}E(u(t))^{s_{c}}\\
&=\|u_{0}\|_{L^{2}}^{2(1-s_{c})}\Big(\frac{1}{2}\|H^{\frac{1}{2}}u(t)\|_{L^{2}}^{2}-\frac{1}{4}\big\||x|^{-b}|u(t)|^{4}\big\|_{L^{1}}\Big)^{s_{c}}\\
&\geq \|u_{0}\|_{L^{2}}^{2(1-s_{c})}\Big(\frac{1}{2}\|H^{\frac{1}{2}}u(t)\|_{L^{2}}^{2}
-\frac{C_{GN}}{4}\|u_{0}\|_{L^{2}}^{1-b}\|H^{\frac{1}{2}}u(t)\|_{L^{2}}^{3+b}\Big)^{s_{c}}\\
&= \Big(\frac{1}{2}\|u_{0}\|_{L^{2}}^{\frac{2(1-s_{c})}{s_{c}}}\|H^{\frac{1}{2}}u(t)\|_{L^{2}}^{2}
-\frac{1}{(3+b)\mathcal{K}}\|u_{0}\|_{L^{2}}^{\frac{2(1-s_{c})(1+s_{c})}{s_{c}}}\|H^{\frac{1}{2}}u(t)\|_{L^{2}}^{2(1+s_{c})}\Big)^{s_{c}}.
\end{split}
\end{align}
Let $f(x)=\frac{1}{2}x^{2}-\frac{1}{(3+b)\mathcal{K}}x^{2(1+s_{c})}$, then $f'(x)=x(1-\frac{1}{\mathcal{K}}x^{2s_{c}})$. Thus, $f'(x)=0$ when
$x_{0}=0$ and $x_{1}=\mathcal{K}^{\frac{1}{2s_{c}}}$. The graph of $f$ is concave for $x\geq 0$ and it has a local minimum at $x_{0}$ and
a local maximum at $x_{1}$. Let $h(t)=\|u_{0}\|_{L^{2}}^{\frac{1-s_{c}}{s_{c}}}\|H^{\frac{1}{2}}u(t)\|_{L^{2}}$. The condition \eqref{4.1}, \eqref{4.6}
and \eqref{4.9}
imply that $M(u_{0})^{1-s_{c}}E(u_{0})^{s_{c}}<(f(\mathcal{K}^{\frac{1}{2s_{c}}}))^{s_{c}}=(\frac{s_{c}}{3+b})^{s_{c}}\mathcal{K}=\mathcal{E}$.
This combined with \eqref{4.11} gives that
$$
f(h(t))<f(x_{1}).
$$
If initially $h(0)<x_{1}$, then by the continuity of $h(t)$ in $t$, we have \eqref{4.3} for any time $t$ belonging to the maximal time interval of existence.
In particular, the $H^{1}-$norm of the solution $u$ is bounded, which proves the global existence in this case. Similarly, if $h(0)>x_{1}$, we have
\eqref{4.5} for any time $t$ belonging to the maximal time interval of existence.
\end{proof}

The next two lemmas provide some additional properties for the solution $u$  under the
hypotheses \eqref{4.1} and \eqref{4.2} of Theorem \ref{th4.1}.
These lemmas
 will be needed in the proof of Theorem \ref{th1.2} through a virial-type estimate, which will be established  in the last
two sections.

\begin{lemma}\label{lem4.2}
In the situation $(i)$ of Theorem \ref{th4.1},
if $u$ is a solution of the problem \eqref{1.1} with radial initial data $u_0$,
then the following statements hold

(i)
\begin{align}\label{4.12}
2E(u_{0})\leq \|H^{\frac{1}{2}}u(t)\|_{L^{2}}^{2}<\frac{3+b}{s_{c}}E(u_{0}), \; \forall t\in{\bf R}.
\end{align}

(ii)
\begin{align}\label{4.13}
\|u_{0}\|_{L^{2}}^{2(1-s_{c})}\|H^{\frac{1}{2}}u(t)\|_{L^{2}}^{2s_{c}}<\omega\mathcal{K},\; \forall t\in{\bf R},
\end{align}
where $\omega=\frac{M(u_{0})^{1-s_{c}}E(u_{0})^{s_{c}}}{\mathcal{E}}$.

(iii)
\begin{align}\label{4.14}
8\|\nabla u(t)\|_{L^{2}}^2-2(3+b)\big\||x|^{-b}|u(t)|^{4}\big\|_{L^{1}}>
8(1-\omega)\|\nabla u\|_{L^2}^2\sim E(u_{0}), \; \forall t\in{\bf R}.
\end{align}

\end{lemma}

\begin{proof}
(i) By the energy conservation law, we obtain
\begin{align}\label{4.15}
E(u_{0})=E(u(t))=\frac{1}{2}\|H^{\frac{1}{2}}u(t)\|_{L^{2}}^{2}-\frac{1}{4}\big\||x|^{-b}|u(t)|^{4}\big\|_{L^{1}}
\end{align}
By the Gagliardo-Nirenberg inequality (with $C_{GN}=\frac{4}{(3+b)\mathcal{K}}$) and \eqref{4.3}, we obtain
\begin{align}
\begin{split}
\big\||x|^{-b}|u(t)|^{4}\big\|_{L^{1}}&\leq \frac{4}{(3+b)\mathcal{K}}\|u(t)\|_{L^{2}}^{1-b}\|H^{\frac{1}{2}}u(t)\|_{L^{2}}^{3+b}\\
&=\frac{4}{(3+b)\mathcal{K}}\|u(t)\|_{L^{2}}^{2(1-s_{c}}\|H^{\frac{1}{2}}u(t)\|_{L^{2}}^{2s_{c}}\|H^{\frac{1}{2}}u(t)\|_{L^{2}}^{2}\\
&<\frac{4}{(3+b)}\|H^{\frac{1}{2}}u(t)\|_{L^{2}}^{2},
\end{split}
\end{align}
which combing with \eqref{4.15} gives
\begin{align}\label{4.17}
E(u_{0})=E(u(t))>\frac{1}{2}\|H^{\frac{1}{2}}u(t)\|_{L^{2}}^{2}-\frac{1}{(3+b)}\|H^{\frac{1}{2}}u(t)\|_{L^{2}}^{2}
=\frac{s_{c}}{3+b}\|H^{\frac{1}{2}}u(t)\|_{L^{2}}^{2}.
\end{align}
On the other hand, it is obvious that
\begin{align}\label{4.18}
E(u_{0})=E(u(t))\leq \frac{1}{2}\|H^{\frac{1}{2}}u(t)\|_{L^{2}}^{2}.
\end{align}
Connecting \eqref{4.17} with \eqref{4.18} gives \eqref{4.12}.

(ii) By the second inequality in (i),
\begin{align}\label{4.19}
\|H^{\frac{1}{2}}u(t)\|_{L^{2}}^{2s_{c}}<\Big(\frac{3+b}{s_{c}}\Big)^{s_{c}}E(u_{0})^{s_{c}}.
\end{align}
Multiplying both sides of \eqref{4.19} by $M(u_{0})^{1-s_{c}}=\|u_{0}\|_{L^{2}}^{2(1-s_{c})}$ and using \eqref{4.6} and \eqref{4.9} yield that
\begin{align}\label{4.20}
\begin{split}
\|u_{0}\|_{L^{2}}^{2(1-s_{c})}\|H^{\frac{1}{2}}u(t)\|_{L^{2}}^{2s_{c}}
&<\Big(\frac{3+b}{s_{c}}\Big)^{s_{c}}E(u_{0})^{s_{c}}\|u_{0}\|_{L^{2}}^{2(1-s_{c})}\\
&=\Big(\frac{3+b}{s_{c}}\Big)^{s_{c}}\mathcal{E}\frac{E(u_{0})^{s_{c}}\|u_{0}\|_{L^{2}}^{2(1-s_{c})}}{\mathcal{E}}\\
&=\omega\mathcal{K}.
\end{split}
\end{align}

(iii) If $V\geq 0$, using the "free" Gagliardo-Nirenberg inequality, \eqref{4.13}, $\|H^{\frac{1}{2}}u(t)\|_{L^{2}}^{2}\geq \|\nabla u(t)\|_{L^{2}}^{2}$,
the equivalence norm \eqref{2.2} and \eqref{4.12} successively gives
\begin{align}\label{4.21}
\begin{split}
8\|\nabla u(t)\|_{L^{2}}^2-2(3+b)\big\||x|^{-b}|u(t)|^{4}\big\|_{L^{1}}
&\geq 8\|\nabla u(t)\|_{L^{2}}^2-\frac{8}{\mathcal{K}}\|u(t)\|_{L^{2}}^{2(1-s_{c})}\|\nabla u(t)\|_{L^{2}}^{2(1+s_{c})}\\
&\geq 8\|\nabla u(t)\|_{L^{2}}^2-\frac{8}{\mathcal{K}}\|u(t)\|_{L^{2}}^{2(1-s_{c})}\|H^{\frac{1}{2}} u(t)\|_{L^{2}}^{2s_{c}}\|\nabla u(t)\|_{L^{2}}^{2}\\
&>8(1-\omega)\|\nabla u\|_{L^2}^2\sim \|H^{\frac{1}{2}} u(t)\|_{L^{2}}^{2}\sim E(u_{0}).
\end{split}
\end{align}

If $V\leq 0$, using Gagliardo-Nirenberg inequality, \eqref{4.13}, $\|H^{\frac{1}{2}}u(t)\|_{L^{2}}^{2}\leq \|\nabla u(t)\|_{L^{2}}^{2}$,
the equivalence norm \eqref{2.2} and \eqref{4.12} successively gives
\begin{align*}
\begin{split}
8\|\nabla u(t)\|_{L^{2}}^2-2(3+b)\big\||x|^{-b}|u(t)|^{4}\big\|_{L^{1}}
&\geq 8\|\nabla u(t)\|_{L^{2}}^2-\frac{8}{\mathcal{K}}\|u(t)\|_{L^{2}}^{2(1-s_{c})}\|H^{\frac{1}{2}} u(t)\|_{L^{2}}^{2(1+s_{c})}\\
&\geq 8\|\nabla u(t)\|_{L^{2}}^2-\frac{8}{\mathcal{K}}\|u(t)\|_{L^{2}}^{2(1-s_{c})}\|H^{\frac{1}{2}} u(t)\|_{L^{2}}^{2s_{c}}\|\nabla u(t)\|_{L^{2}}^{2}\\
&>8(1-\omega)\|\nabla u\|_{L^2}^2\sim \|H^{\frac{1}{2}} u(t)\|_{L^{2}}^{2}\sim E(u_{0}).
\end{split}
\end{align*}
\end{proof}

Finally, we give the result about existence of wave operators, which will be used to established the scattering theory.
Before giving its statement, we need Lemma \ref{lemFG}, which is namely Lemma 5.2 in \cite{FG1} with $e^{-itH}$
in place of $e^{it\Delta}$ since the operator $e^{-itH}$ satisfies the same dispersive estimates (see Lemma 2.1 in \cite{H}).

\begin{lemma}\label{lemFG}
Let $0<b<1$. If $f$ and $g\in H^1({\bf R}^{3})$, then there exists some $\frac{12}{3-b}<r<6$ such that

(i) $\big\||x|^{-b}|f|^{3}g\big\|_{L^{1}}\leq c\|f\|_{L^{4}}^{3}\|g\|_{L^{4}}+c\|f\|_{L^{r}}^{3}\|g\|_{L^{r}}$;

(ii) $\big\||x|^{-b}|f|^{3}g\big\|_{L^{1}}\leq c\|f\|_{H^{1}}^{3}\|g\|_{H^{1}}$;

(iii) $\lim_{|t|\rightarrow+\infty}\big\||x|^{-b}|e^{-itH}f|^{3}g\big\|_{L_{x}^{1}}=0$.
\end{lemma}
\vskip0.5cm

\begin{proposition}\label{pro4.3}
If $V$ is radially symmetric and satisfies \eqref{1.2} and \eqref{1.3}, $V\geq 0$ or $V\leq 0$, and $0<b<1$.
 Suppose radial function $ \psi^\pm\in H^1({\bf R}^{3})$ and
\begin{align}\label{4.22}
(\frac{1}{2}\|H^{\frac{1}{2}}\psi^\pm\|_{L^{2}}^2)^{s_{c}}\|\psi^\pm\|_{L^{2}}^{2(1-s_c)}<\mathcal{E}.
\end{align}
Then there exists a unique radial function $v_0\in H^1({\bf R}^{3})$ such that the solution $v$ of
\eqref{1.1} with initial data $v_0$ obeys the assumptions \eqref{4.1} and \eqref{4.2} and satisfies
\begin{align}\label{4.23}
\lim_{t\rightarrow\pm\infty}\|v(t)-e^{itH}\psi^\pm\|_{H^1({\bf R}^{3})}=0.
\end{align}
Moreover,
\begin{align}\label{4.24}
\|v(t)\|_{S(\dot{H}^{s_{c}})}<\infty\;\text{and}\; \|\langle\nabla\rangle v\|_{S(L^{2})}<\infty.
\end{align}
\end{proposition}

\begin{proof}
 Similar to the small data theory Proposition \ref{pro2.5},
we can solve the integral equation
\begin{align}\label{4.25}
v(t)=e^{-itH}\psi^+ -i\int_t^\infty e^{-i(t-s)H}(|x|^{-b}|v|^{2}v)(s)ds
\end{align}
for $t\geq T$ with $T$ large.

 In fact, there exists some large $T$ such that
$\|e^{-itH}\psi^+\|_{S(\dot{H}^{s_{c}}, [T,\infty))}\leq \delta_{sd},$
where $\delta_{sd}$ is defined by Proposition \ref{pro2.5}.
Then, the same arguments as in Proposition \ref{pro2.5} give a unique solution
$v\in C([T,\infty),H^1)$ of \eqref{4.25}.
Moreover, we also have
\begin{align}\label{4.26}
\|v\|_{S(\dot{H}^{s_{c}}, [T,\infty))}\leq2\delta_{sd},\quad
{\rm and}\quad
 \|\langle\nabla\rangle v\|_{S(L^{2}, [T, +\infty))} \leq 2\|\psi^{+}\|_{H^{1}}.
\end{align}
Thus by \eqref{2.1}, \eqref{2.2}, \eqref{2.13}, \eqref{2.14} and \eqref{2.17},
\begin{align}\label{4.27}
\begin{split}
&\| v-e^{-itH}\psi^+\|_{L_{[T,\infty)}^{\infty}H_{x}^{1}}
=\left\|\int_t^\infty e^{-i(t-s)H}(|x|^{-b}|v|^{2}v)(s)ds\right\|_{L_{[T,\infty)}^{\infty}H_{x}^{1}}\\
&\lesssim \|\langle\nabla\rangle(|x|^{-b}|u|^{2}u)\|_{L_{[T,+\infty)}^{2}L_{x}^{\frac{6}{5}}}
\lesssim \||\nabla|^{s_{c}}u\|_{L_{[T,+\infty)}^{4}L_{x}^{3}}\|u\|_{L_{[T,+\infty)}^{\frac{4}{1-b}}L_{x}^{6}}
\|\langle\nabla\rangle u\|_{L_{[T,+\infty)}^{\frac{4}{b}}L_{x}^{\frac{6}{3-b}}}\\
&\lesssim \|\psi^{+}\|_{H^{1}}^{2}\delta_{sd}.
\end{split}
\end{align}
As $T\rightarrow \infty$, $\delta_{sd}>0$ can be chosen small enough, so we have
\begin{align}\label{4.28}
\| v-e^{-itH}\psi^+\|_{L_{[T,\infty)}^{\infty}H_{x}^{1}} \rightarrow 0
\ \  {\rm as} ~T\rightarrow \infty,
\end{align}
which implies $v(t)-e^{-itH}\psi^+\rightarrow0 $ in $H^1({\bf R}^{3})$ as $t\rightarrow+\infty$. Hence, we have
\begin{align}\label{4.29}
\lim_{t\rightarrow+\infty}\|u(t)\|_{L_{x}^{2}}=\|\psi^{+}\|_{L_{x}^{2}}
\end{align}
and
\begin{align}\label{4.30}
\lim_{t\rightarrow+\infty}\|H^{\frac{1}{2}}u(t)\|_{L_{x}^{2}}=\|H^{\frac{1}{2}}\psi^{+}\|_{L_{x}^{2}}.
\end{align}
By the mass conservation, we have $\|u(t)\|_{L_{x}^{2}}=\|u(T)\|_{L_{x}^{2}}$ for all $t\geq T$.
So from \eqref{4.29}, we deduce $\|u(T)\|_{L_{x}^{2}}=\|\psi^{+}\|_{L_{x}^{2}}$.

On the other hand, by Lemma \ref{lemFG}
\begin{align}\label{4.31}
\begin{split}
\big\||x|^{-b}|u(t)|^{4}\big\|_{L_{x}^{1}}
&\lesssim \big\||x|^{-b}|u(t)-e^{-itH}\psi^{+}|^{4}\big\|_{L_{x}^{1}}+\big\||x|^{-b}|e^{-itH}\psi^{+}|^{4}\big\|_{L_{x}^{1}}\\
&\lesssim \|u(t)-e^{-itH}\psi^{+}\|_{H_{x}^{1}}^{4}+\big\||x|^{-b}|e^{-itH}\psi^{+}|^{4}\big\|_{L_{x}^{1}}.
\end{split}
\end{align}
Using \eqref{4.28} and Lemma \ref{lemFG} again gives
\begin{align}\label{4.32}
\lim_{t\rightarrow+\infty}\big\||x|^{-b}|u(t)|^{4}\big\|_{L_{x}^{1}}=0.
\end{align}

Thus combining \eqref{4.30} and \eqref{4.32}, we obtain that
\begin{align}\label{4.33}
\begin{split}
M(v(T))^{ 2-s_c }E(v(T))^{s_{c}}
&= \lim_{t\rightarrow+\infty}M(v(t))^{ 2-s_c}E(v(t))^{s_{c}}\\
&=\|\psi^+\|_{L^{2}}^{2(2-s_c)}\frac{1}{2^{s_{c}}}\|H^{\frac{1}{2}}\psi^\pm\|_{L^2}^{2s_{c}}<\mathcal{E}.
\end{split}
\end{align}
Moreover, we note that
\begin{align}\label{4.34}
&\lim_{t\rightarrow\infty}\|v(t)\|_{L_{x}^2}^{2(2-s_c)}\|H^{\frac{1}{2}} v(t)\|_{L_{x}^2}^{2s_{c}}
=\|\psi^+\|_{L_{x}^2}^{(2-s_c)}\|H^{\frac{1}{2}}\psi^+\|_{L_{x}^2}^{2s_{c}}\nonumber\\
&<2^{s_{c}}\mathcal{E}
=(\frac{2s_{c}}{3+b})^{s_{c}}\mathcal{K}<\mathcal{K}
\end{align}

Therefore, for $T$ large enough, $v(T)$ satisfies \eqref{4.1} and \eqref{4.2}, which, due to Theorem \ref{4.1}, implies that $v(t)$ exists globally in $H^{1}({\bf R}^{3})$.
Thus, we can evolve $v(t)$ from $T$ back to the initial time 0. By the same way, we can show \eqref{4.23} for the negative time. In addition, \eqref{4.26}
combined  with local theory implies \eqref{4.24}.

\end{proof}
 \vskip0.5cm

\medskip

\section{Existence and compactness of a critical element}
\setcounter{equation}{0}

\begin{definition}\label{def5.1}
We say that $SC(u_0)$ holds if for radial $u_0\in H^1({\bf R}^{3})$ satisfying
$\|u_{0}\|^{2(1-s_c)}_{L^{2}}\|H^{\frac{1}{2}}
u_{0}\|_{L^{2}}^{2s_{c}}<\mathcal{K}$ and
$E(u_{0})^{s_{c}}M(u_{0})^{1-s_c}<\mathcal{E}$,
 the corresponding solution $u$  of \eqref{1.1} with the maximal
interval of existence $I=(-\infty,+\infty)$ satisfies
\begin{align}\label{5.1}
\|u\|_{S(\dot{H}^{s_{c}})}<+\infty.
\end{align}
\end{definition}

We first claim  that  there exists $\delta>0$ such that if
$E(u_{0})^{s_{c}}M(u_{0})^{1-s_c}<\delta$ and
$\|u_{0}\|^{2(1-s_c)}_{L^{2}}\|H^{\frac{1}{2}}
u_{0}\|_{L^{2}}^{2s_{c}}<\mathcal{K},$ then
\eqref{5.1} holds.
In fact, the Strichartz estimate \eqref{2.1},
the norm equivalence \eqref{2.2} and \eqref{4.12}, we have
\begin{align}\label{5.2}
\begin{split}
\|e^{itH}u_{0}\|_{S(\dot{H}^{s_{c}})}^{2}
&\lesssim \||\nabla|^{s_{c}}u_{0}\|_{L^{2}}^{2}
\lesssim\|u_{0}\|_{L^{2}}^{2(1-s_{c})}\|\nabla u_{0}\|_{L^{2}}^{2s_{c}}\\
&\sim \|u_{0}\|_{L^{2}}^{2(1-s_{c})}\|H^{\frac{1}{2}} u_{0}\|_{L^{2}}^{2s_{c}}
\sim E(u_{0})^{s_{c}}M(u_{0})^{1-s_{c}}.
\end{split}
\end{align}
So if $E(u_{0})^{s_{c}}M(u_{0})^{1-s_c}<\delta$ with sufficiently small $\delta>0$, 
we have that  $\|e^{-itH}u_{0}\|_{S(\dot{H}^{s_{c}})}\leq \delta_{sd}$. Therefore, 
it follows from Proposition \ref{pro2.5} that \eqref{5.1} holds for all
sufficiently small $\delta>0$, which implies scattering by  Proposition \ref{pro2.6}.

Now for each $\delta>0$, we define the set $S_\delta$ to be the collection
of all such initial data in $H^1$ :
\begin{align}\label{5.3}
S_\delta=\{u_0\in H^1({\bf R}^{3}):\ \  E(u_{0})^{s_{c}}M(u_{0})^{1-s_c}<\delta \ \  and \ \
\|u_{0}\|^{2(1-s_c)}_{L^{2}}\|H^{\frac{1}{2}}
u_{0}\|_{L^{2}}^{2s_{c}}<\mathcal{K}\}.
\end{align}
We also define
\begin{align}\label{5.4}
\mathcal{E}_c=\sup\{\delta:\ \ u_0\in S_\delta\Rightarrow SC(u_0)\ \  holds \}.
\end{align}
If $\mathcal{E}_c=\mathcal{E}$, then we are done. Thus we assume now
\begin{align}\label{5.5}
\mathcal{E}_c<\mathcal{E}.
\end{align}
Our goal in this section is to show the existence of an $H^1$-
solution $u_c$ of \eqref{1.1} with the initial data $u_{c,0}$ such that
\begin{align}\label{5.6}
\|u_{c,0}\|^{2(1-s_c)}_{L^{2}}\|H^{\frac{1}{2}}
u_{c,0}\|_{L^{2}}^{2s_{c}}<\mathcal{K},
\end{align}
\begin{align}\label{5.7}
E(u_{c,0})^{s_{c}}M(u_{c,0})^{1-s_c}=\mathcal{E}_c
\end{align}
and
$ SC(u_{c,0})$ does not hold. Moreover, we  show that if
$\|u_c\|_{S(\dot{H}^{s_{c}})}=\infty$, then
$K=\{u_c(x,t)|t\in {\bf R}\}$ is precompact in  $H^1({\bf R}^{3})$.

Prior to fulfilling  our main task, we first state
 the linear profile decomposition  associated with a perturbed linear propagator $e^{itH_{r_{n}}}$,
 with $$H_{r_{n}}=-\Delta+V_{r_n}, \ \ \ \ V_{r_n}(x)=\frac1{r_n^2}V(\frac x{r_n}),$$
 which was established by Hong \cite{H}.
 The profile decomposition
associated with the free linear propagator $e^{it\Delta}$ \cite{DHR, HR} was established by using the concentration
compactness principle in the spirit of Keraani \cite{Ker} and Kenig and Merle \cite{Kenig}.

\begin{proposition}\label{pro5.2}
If $V$  is radial and satisfies \eqref{1.2}, \eqref{1.3} and $|x||\nabla V|\in L^{\frac{3}{2}}$.
Let $\phi_{n}(x)$ be radial and uniformly
bounded  in $H^{1}({\bf R}^{3})$, and  $r_{n}=1, r_{n}\rightarrow 0$ or $r_{n}\rightarrow\infty$.
Then for each $M$ there exists a
subsequence of $\phi_{n}$, which is denoted by itself, such that the
following statements hold. \\
(1) For each $1\leq j\leq M$, there exists
(fixed in n) a profile $\psi^{j}(x)$ in $H^1({\bf R}^{3})$ and
 a sequence (in $n$) of time
shifts $t_{n}^{j}$, and there exists a sequence (in $n$) of
remainders $W_{n}^{M}(x)$ in $H^1({\bf R}^{3})$  such that
\begin{align}\label{5.8}
\phi_{n}(x)=\sum_{j=1}^{M}e^{it_{n}^{j}H_{r_{n}}}\psi^{j}(x)+W_{n}^{M}(x).
\end{align}
(2) The time  sequences have a pairwise divergence property, i.e.,  for $1\leq j\neq k\leq M$,
\begin{align}\label{5.9}
\lim_{n\rightarrow+\infty}
|t_{n}^{j}-t_{n}^{k}|=+\infty.
\end{align}
(3) The remainder sequence has the following asymptotic smallness
property:
\begin{align}\label{5.10}
\lim_{M\rightarrow+\infty}[\lim_{n\rightarrow+\infty}\|e^{-itH_{r_{n}}}W_{n}^{M}\|_{S(\dot{H}^{s_{c}})}]=0.
\end{align}
(4) For each fixed $M$, we have the asymptotic
Pythagorean expansion as follows
\begin{align}\label{5.11}
\|\phi_{n}\|_{L^{2}}^{2}
=\sum_{j=1}^{M}\|\psi^{j}\|_{L^{2}}^{2}+\|W_{n}^{M}\|_{L^{2}}^{2}+o_{n}(1),
\end{align}
\begin{align}\label{5.12}
\|H_{r_{n}}^{\frac{1}{2}}\phi_{n}\|_{L^{2}}^{2}
=\sum_{j=1}^{M}\|H_{r_{n}}^{\frac{1}{2}}\psi^{j}\|_{L^{2}}^{2}
+\|H_{r_{n}}^{\frac{1}{2}}W_{n}^{M}\|_{L^{2}}^{2}+o_{n}(1),
\end{align}
where $o_{n}(1)\rightarrow0$ as $n\rightarrow+\infty$.
\end{proposition}

\begin{proof} The proof is similar to that of Proposition 5.1 in \cite{H}. For the sake of completeness, we give its detail
here. Let's first consider the case $r_{n}\rightarrow 0$ or $r_{n}\rightarrow\infty$.
According to the profile decomposition associated with $e^{it\Delta}$ (see Lemma 5.2 in \cite{HR}), 
there exists a subsequence of $\phi_{n}$, which
is still denoted by itself and satisfies all the properties in Proposition \ref{pro5.2} with $V=0$. For example, \eqref{5.8} 
with $V=0$ is namely
\begin{align}\label{5.13}
\phi_{n}(x)=\sum_{j=1}^{M}e^{-it_{n}^{j}\Delta}\psi^{j}(x)+W_{n}^{M}(x)
\end{align}

In order to get the form of \eqref{5.8}, we can rewrite \eqref{5.13} as
\begin{align}\label{5.14}
\phi_{n}(x)=\sum_{j=1}^{M}e^{it_{n}^{j}H_{r_{n}}}\psi^{j}(x)+{\widetilde W}_{n}^{M}(x),
\end{align}
where
\begin{align}\label{5.15}
{\widetilde W}_{n}^{M}(x)=W_{n}^{M}(x)
+\sum_{j=1}^{M}\Big(e^{-it_{n}^{j}\Delta}\psi^{j}(x)-e^{it_{n}^{j}H_{r_{n}}}\psi^{j}(x)\Big).
\end{align}
Now we start to verify that \eqref{5.14} satisfies the properties \eqref{5.9}-\eqref{5.12}. It's obvious that \eqref{5.9} is true,
so let's look at \eqref{5.10}. We note that $u(t)=e^{it\Delta}u_{0}$ is a solution to the integral equation
\begin{align}\label{5.16}
u(t)=e^{it(\Delta-V_{r_{n}})}u_{0}-i\displaystyle\int_{0}^{t}e^{i(t-s)(\Delta-V_{r_{n}})}(V_{r_{n}}u(s))ds.
\end{align}
Substituting $u_0=W_{n}^{M}$ in the formula \eqref{5.16} yields that
\begin{align}\label{5.17}
\begin{split}
\|e^{-itH_{r_{n}}}W_{n}^{M}\|_{S(\dot{H}^{s_{c}})}
&\leq \|e^{it\Delta}W_{n}^{M}\|_{S(\dot{H}^{s_{c}})}
+\|\displaystyle\int_{0}^{t}e^{i(t-s)(\Delta-V_{r_{n}})}(V_{r_{n}}e^{is\Delta}W_{n}^{M})ds\|_{S(\dot{H}^{s_{c}})}\\
&\lesssim \|e^{it\Delta}W_{n}^{M}\|_{S(\dot{H}^{s_{c}})}
+\|V_{r_{n}}e^{it\Delta}W_{n}^{M}\|_{L_{t}^{\frac{4}{1-b}}L_{x}^{\frac{6}{5}}}\\
&\lesssim \|e^{it\Delta}W_{n}^{M}\|_{S(\dot{H}^{s_{c}})}
+\|V_{r_{n}}\|_{L^{\frac{3}{2}}}\|e^{it\Delta}W_{n}^{M}\|_{L_{t}^{\frac{4}{1-b}}L_{x}^{6}}\\
&=(1+\|V\|_{L^{\frac{3}{2}}})\|e^{it\Delta}W_{n}^{M}\|_{S(\dot{H}^{s_{c}})}\rightarrow 0,
\end{split}
\end{align}
as $n\rightarrow\infty$ and $M\rightarrow\infty$.

Also applying \eqref{5.16}, we obtain
\begin{align}\label{5.18}
\begin{split}
&\|e^{-itH_{r_{n}}}(e^{-it_{n}^{j}\Delta}\psi^{j}-e^{it_{n}^{j}H_{r_{n}}}\psi^{j})\|_{S(\dot{H}^{s_{c}})}\\
&=\|\displaystyle\int_{-t_{n}^{j}}^{0}
e^{-i(t+t_{n}^{j}+s)H_{r_{n}}}(V_{r_{n}}e^{is\Delta}\psi^{j})ds\|_{S(\dot{H}^{s_{c}})}\\
&\lesssim \|V_{r_{n}}e^{-it\Delta}\psi^{j}\|_{L_{t}^{\frac{4}{1-b}}L_{x}^{\frac{6}{5}}}\rightarrow 0,
\end{split}
\end{align}
as $n\rightarrow\infty$, where in the last step we have used
\begin{align}\label{5.19}
\|V_{r_{n}}e^{-it\Delta}\psi^{j}\|_{L_{t}^{\frac{4}{1-b}}L_{x}^{\frac{6}{5}}}
\lesssim \|V_{r_{n}}\|_{L^{\frac{3}{2}}}\|e^{-it\Delta}\psi^{j}\|_{L_{t}^{\frac{4}{1-b}}L_{x}^{6}}
\lesssim \|V\|_{L^{\frac{3}{2}}}\|\psi^{j}\|_{\dot{H}^{s_{c}}}.
\end{align}
and the condition $r_{n}\rightarrow 0$ or $\infty$. So it follows from \eqref{5.17} and \eqref{5.18} that
 ${\widetilde W}_{n}^{M}(x)$ in \eqref{5.15}
satisfies the property \eqref{5.10}.

To get \eqref{5.11}, it suffices to prove
\begin{align}\label{5.20}
\|{\widetilde W}_{n}^{M}\|_{L^{2}}^{2}=\|W_{n}^{M}\|_{L^{2}}^{2}+o_{n}(1).
\end{align}
It follows from the expression of ${\widetilde W}_{n}^{M}(x)$ \eqref{5.15} that
\begin{align}\label{5.21}
\begin{split}
&\|{\widetilde W}_{n}^{M}\|_{L^{2}}^{2}=\|W_{n}^{M}\|_{L^{2}}^{2}
+2\sum_{j=1}^{M}\langle W_{n}^{M}, e^{-it_{n}^{j}\Delta}\psi^{j}-e^{it_{n}^{j}H_{r_{n}}}\psi^{j}\rangle\\
&+2\sum_{k\neq j}\langle e^{-it_{n}^{k}\Delta}\psi^{j}-e^{it_{n}^{k}H_{r_{n}}}\psi^{j},
e^{-it_{n}^{j}\Delta}\psi^{j}-e^{it_{n}^{j}H_{r_{n}}}\psi^{j}\rangle\\
&+\sum_{j=1}^{M}\|e^{-it_{n}^{j}\Delta}\psi^{j}-e^{it_{n}^{j}H_{r_{n}}}\psi^{j}\|_{L^{2}}^{2},
\end{split}
\end{align}
from which, we only need to show that
\begin{align}\label{5.22}
\|e^{-it_{n}^{j}\Delta}\psi^{j}-e^{it_{n}^{j}H_{r_{n}}}\psi^{j}\|_{L^{2}}\rightarrow 0,
\end{align}
as $n\rightarrow \infty$.

In fact,
\begin{align}\label{5.23}
\begin{split}
\|e^{-it_{n}^{j}\Delta}\psi^{j}-e^{it_{n}^{j}H_{r_{n}}}\psi^{j}\|_{L^{2}}
&=\|\displaystyle\int_{-t_{n}^{j}}^{0}e^{-i(t_{n}^{j}+s)H_{r_{n}}}(V_{r_{n}}e^{is\Delta}\psi^{j})ds\|_{L^{2}}\\
&\lesssim \|V_{r_{n}}e^{it\Delta}\psi^{j}\|_{L_{t}^{2}L_{x}^{\frac{6}{5}}}\rightarrow 0,
\end{split}
\end{align}
as $n\rightarrow \infty$. where the last step follows from
\begin{align}\label{5.24}
\|V_{r_{n}}e^{it\Delta}\psi^{j}\|_{L_{t}^{2}L_{x}^{\frac{6}{5}}}
\lesssim \|V_{r_{n}}\|_{L^{\frac{3}{2}}}\|e^{it\Delta}\psi^{j}\|_{L_{t}^{2}L_{x}^{6}}
\lesssim \|V\|_{L^{\frac{3}{2}}}\|\psi^{j}\|_{L^{2}}.
\end{align}
and the condition $r_{n}\rightarrow 0$ or $\infty$. Thus, we complete the proof of \eqref{5.11}.

Now we turn to \eqref{5.12}. Since
\begin{align*}
\|H_{r_{n}}^{\frac{1}{2}}f_{n}\|_{L^{2}}^{2}=\|\nabla f_{n}\|_{L^{2}}^{2}+\langle V_{r_{n}}f_{n}, f_{n}\rangle
\end{align*}
and
$$
|\langle V_{r_{n}}f_{n}, f_{n}\rangle|\lesssim \|V_{r_{n}}\|_{L^{\frac{3}{2}}}\|f_{n}\|_{L^{6}}^{2}
\lesssim \|V\|_{L^{\frac{3}{2}}}\|\nabla f_{n}\|_{L^{2}}^{2},
$$
we have
\begin{align}\label{5.25}
\|H_{r_{n}}^{\frac{1}{2}}f_{n}\|_{L^{2}}^{2}=\|\nabla f_{n}\|_{L^{2}}^{2}+o_{n}(1),
\end{align}
provided that $\|\nabla f_{n}\|_{L^{2}}$ is uniformly bounded. Hence, applying \eqref{5.25} with $\phi_{n}$, $\phi^{j}$
and ${\widetilde W}_{n}^{M}$ and using the asymptotic Pythagorean expansion associated with the free linear propagator,
we find that \eqref{5.12} can be deduced from the following expression
\begin{align}\label{5.26}
\|\nabla {\widetilde W}_{n}^{M}\|_{L^{2}}^{2}=\|\nabla W_{n}^{M}\|_{L^{2}}^{2}+o_{n}(1).
\end{align}

As in the proof of \eqref{5.20}, it suffices to prove
\begin{align}\label{5.27}
\|\nabla (e^{-it_{n}^{j}\Delta}\psi^{j}-e^{it_{n}^{j}H_{r_{n}}}\psi^{j})\|_{L^{2}}\rightarrow 0,
\end{align}
as $n\rightarrow \infty$. Indeed, using \eqref{2.1} and \eqref{2.2}, we have
\begin{align}\label{5.28}
\begin{split}
\|\nabla (e^{-it_{n}^{j}\Delta}\psi^{j}-e^{it_{n}^{j}H_{r_{n}}}\psi^{j})\|_{L^{2}}
&\leq \|H^{\frac{1}{2}}\displaystyle\int_{-t_{n}^{j}}^{0}e^{-i(t_{n}^{j}+s)H_{r_{n}}}(V_{r_{n}}e^{isH_{0}}\psi^{j})ds\|_{L^{2}}\\
&\lesssim \|\nabla (V_{r_{n}}e^{isH_{0}}\psi^{j})\|_{L_{t}^{2}L_{x}^{\frac{6}{5}}}\rightarrow 0,
\end{split}
\end{align}
as $n\rightarrow \infty$, where the last step follows from
\begin{align}\label{5.29}
\begin{split}
\|\nabla (V_{r_{n}}e^{is\Delta}\psi^{j})\|_{L_{t}^{2}L_{x}^{\frac{6}{5}}}
&\lesssim \big\||x||\nabla V_{r_{n}}|\big\|_{L^{\frac{3}{2}}}\big\||x|^{-1}e^{is\Delta}\psi^{j}\big\|_{L_{t}^{2}L_{x}^{6}}
+\| V_{r_{n}}\|_{L^{\frac{3}{2}}}\|\nabla e^{is\Delta}\psi^{j}\|_{L_{t}^{2}L_{x}^{6}}\\
&\lesssim  \Big(\big\||x||\nabla V|\big\|_{L^{\frac{3}{2}}}+\| V\|_{L^{\frac{3}{2}}}\Big)
\|\Delta e^{is\Delta}\psi^{j}\|_{L_{t}^{2}L_{x}^{6}}\\
&\lesssim  \Big(\big\||x||\nabla V|\big\|_{L^{\frac{3}{2}}}+\| V\|_{L^{\frac{3}{2}}}\Big)\|\psi^j\|_{H^{1}}.
\end{split}
\end{align}

Now Let's consider the other case $r_{n}=1$. Using \eqref{5.13} again gives
\begin{align}\label{5.30}
\phi_{n}(x)=\sum_{j=1}^{M}e^{-it_{n}^{j}\Delta}\psi^{j}(x)+W_{n}^{M}(x).
\end{align}
If, on one hand, $t_{n}^{j}\rightarrow\infty$, by Lemma \ref{lem2.1}, there exists $\tilde{\psi}^{j}\in H^{1}({\bf R}^{3})$ such that
$$
\|e^{-it_{n}^{j}\Delta}\psi^{j}-e^{it_{n}^{j}H}\tilde{\psi}^{j}\|_{H^{1}}\rightarrow 0.
$$
If, on the other hand,
$t_{n}^{j}=0$, we define $\tilde{\psi}^{j}=\psi^{j}$. To sum up, in either case, we obtain a new profile $\tilde{\psi}^{j}$
for the given $\psi^{j}$ such that
\begin{align}\label{5.31}
\|e^{-it_{n}^{j}\Delta}\psi^{j}-e^{-it_{n}^{j}H}\tilde{\psi}^{j}\|_{H^{1}}\rightarrow 0,
\ {\rm as}\ \ n\rightarrow+\infty.
\end{align}
In order to get the form of \eqref{5.10}, we can rewrite \eqref{5.8} as
\begin{align}\label{5.32}
\phi_{n}(x)=\sum_{j=1}^{M}e^{it_{n}^{j}H}\tilde{\psi}^{j}(x)+{\widetilde W}_{n}^{M}(x),
\end{align}
where
\begin{align}\label{5.33}
{\widetilde W}_{n}^{M}(x)=W_{n}^{M}(x)+\sum_{j=1}^{M}\Big(e^{-it_{n}^{j}\Delta}\psi^{j}(x)-e^{it_{n}^{j}H}\tilde{\psi}^{j}(x)\Big)
\end{align}
Here we only give the proof of \eqref{5.10}, since all the proofs of \eqref{5.11}-\eqref{5.12} can be obtained by following the same argument
in the case $r_{n}\rightarrow 0$ or $\infty$ and using \eqref{5.31}.  Indeed, \eqref{5.17} with $r_{n}=1$ is still valid, which yields
\begin{align}\label{5.34}
\lim_{M\rightarrow+\infty}[\lim_{n\rightarrow+\infty}\|e^{itH}W_{n}^{M}\|_{S(\dot{H}^{s_{c}})}]=0.
\end{align}
And using the Strichartz estimate \eqref{2.1} and \eqref{5.31}, we have
\begin{align}\label{5.35}
\begin{split}
\|e^{-itH}(e^{-it_{n}^{j}\Delta}\psi^{j}(x)-e^{it_{n}^{j}H}\tilde{\psi}^{j}(x))\|_{S(\dot{H}^{s_{c}})}
&\lesssim \|e^{itH}(e^{-it_{n}^{j}\Delta}\psi^{j}(x)-e^{it_{n}^{j}H}\tilde{\psi}^{j}(x))\|_{\dot{H}^{s_{c}}}\\
&\lesssim \|e^{-it_{n}^{j}\Delta}\psi^{j}(x)-e^{it_{n}^{j}H}\tilde{\psi}^{j}(x)\|_{H^{1}}\rightarrow 0,
\end{split}
\end{align}
as $n\rightarrow\infty$. Putting \eqref{5.34} and \eqref{5.35} together gives \eqref{5.10}, that is,
\begin{align}\label{5.36}
\lim_{M\rightarrow+\infty}[\lim_{n\rightarrow+\infty}\|e^{-itH}{\widetilde W}_{n}^{M}\|_{S(\dot{H}^{s_{c}})}]=0.
\end{align}
\end{proof} \vskip0.5cm

\begin{remark}\label{rem5.3}
We claim that
\begin{align}\label{5.37}
\lim_{M, n\rightarrow\infty}\|W_{n}^{M}\|_{L_{x}^{p}}=0.
\end{align}
where $2<p<6$.

Indeed, when $r_{n}\rightarrow 0$ or $\infty$, it follows from Remark 6.2 in \cite{FG1} (that is, \eqref{5.37} holds when $V=0$)
and \eqref{5.15} that it suffices to
show that
\begin{align}\label{5.38}
\|e^{-it_{n}^{j}\Delta}\psi^{j}-e^{it_{n}^{j}H_{r_{n}}}\psi^{j}\|_{L^{p}}\rightarrow 0,
\end{align}
as $n\rightarrow \infty$,
which is implied by Sobolev embedding, \eqref{5.22} and \eqref{5.27}.

Similarly, when $r_{n}=1$, by Remark 6.2 in \cite{FG1} again and \eqref{5.33}, we only prove that
\begin{align}\label{5.39}
\|e^{-it_{n}^{j}\Delta}\psi^{j}(x)-e^{it_{n}^{j}H}\tilde{\psi}^{j}(x)\|_{L^{p}}\rightarrow 0,
\end{align}
as $n\rightarrow \infty$, which is implied by Sobolev embedding and \eqref{5.31}.

It follows from \eqref{5.37} and Lemma \ref{lemFG} that
\begin{align}\label{5.40}
\lim_{M, n\rightarrow\infty}\big\||x|^{-b}|W_{n}^{M}|^{4}\big\|_{L_{x}^{1}}=0.
\end{align}

\end{remark}
\vskip0.5cm

Next, we shall use Proposition \ref{pro5.2} and Remark \ref{rem5.3} to establish the energy pythagorean expansion.

\begin{lemma}\label{lem5.4} In the situation of Proposition \ref{pro5.2}, we have
\begin{align}\label{5.41}
E_{V_{r_{n}}}(\phi_{n})=\sum_{j=1}^{M}E_{V_{r_{n}}}(e^{it_{n}^{j}H_{r_{n}}}\psi^{j})+E_{V_{r_{n}}}(W_{n}^{M})+o_n(1).
\end{align}
\end{lemma}

\begin{proof}
 According to \eqref{5.11} and \eqref{5.12}, it is sufficient to establish for all $M\geq1$,
\begin{align}\label{5.42}
\big\||x|^{-b}|\phi_{n}|^{4}\big\|_{L_{x}^{1}}=\sum_{j=1}^{M}\big\||x|^{-b}|e^{it_{n}^{j}H_{r_{n}}}\psi^{j}|^{4}\big\|_{L_{x}^{1}}
+\big\||x|^{-b}|W_{n}^{M}|^{4}\big\|_{L_{x}^{1}}+o_n(1).
\end{align}
And, by density. one may set $\psi^{j}\in C_{c}^{\infty}({\bf R}^{3})$.

We Claim that
\begin{align}\label{5.43}
\big\||x|^{-b}|\sum_{j=1}^{M}e^{it_{n}^{j}H_{r_{n}}}\psi^{j}|^{4}\big\|_{L_{x}^{1}}=\sum_{j=1}^{M}\big\||x|^{-b}|e^{it_{n}^{j}H_{r_{n}}}\psi^{j}|^{4}\big\|_{L_{x}^{1}}+o_n(1).
\end{align}
In fact, left hand side is the linear combination of the terms in the form
\begin{align}\label{5.44}
\displaystyle\int_{{\bf R}^{3}}|x|^{-b}e^{it_{n}^{j_{1}}H_{r_{n}}}\psi^{j_{1}}\overline{e^{it_{n}^{j_{2}}H_{r_{n}}}\psi^{j_{2}}}e^{it_{n}^{j_{3}}H_{r_{n}}}\psi^{j_{3}}
\overline{e^{it_{n}^{j_{4}}H_{r_{n}}}\psi^{j_{4}}}dx.
\end{align}
By \eqref{5.9}, there is at least one $j_{k}$ satisfying $t_{n}^{j_{k}}\rightarrow\infty$. Without loss of generality, we assume that $t_{n}^{j_{1}}\rightarrow\infty$.
Using Lemma \ref{lemFG}, the dispersive estimate (see Lemma 2.1 in \cite{H}), Sobolev embedding and the norm equivalence, we have that \eqref{5.44} is bounded by
\begin{align}\label{5.45}
\begin{split}
\Big(&\|e^{it_{n}^{j_{1}}H_{r_{n}}}\psi^{j_{1}}\|_{L_{x}^{4}}+\|e^{it_{n}^{j_{1}}H_{r_{n}}}\psi^{j_{1}}\|_{L_{x}^{r}}\Big)\prod_{k=2,3,4}\|e^{it_{n}^{j_{k}}H_{r_{n}}}\psi^{j_{k}}\|_{H_{x}^{1}}\\
&\lesssim (|t_{n}^{j_{1}}|^{-\frac{3}{4}}\|\psi^{j_{1}}\|_{L_{x}^{\frac{4}{3}}}+|t_{n}^{j_{1}}|^{-\frac{3}{2}(1-\frac{2}{r})}\|\psi^{j_{1}}\|_{L_{x}^{r'}})\|\psi^{j_{2}}\|_{H_{x}^{1}}
\|\psi^{j_{3}}\|_{H_{x}^{1}}\|\psi^{j_{4}}\|_{H_{x}^{1}}\rightarrow 0,
\end{split}
\end{align}
where $\frac{12}{3-b}<r<12$. Hence, \eqref{5.44} tends to zero as $n\rightarrow +\infty$. It follows from \eqref{5.40} that, for any $\varepsilon>0$, there exists $M_{1}\gg 1$ such that
$\big\||x|^{-b}|W_{n}^{M_{1}}|^{4}\big\|_{L_{x}^{1}}\leq\varepsilon$ for all sufficiently large $n$. Hence, we obtain
\begin{equation}\label{5.46}
\begin{split}
\big\||x|^{-b}|\phi_{n}|^{4}\big\|_{L_{x}^{1}}
&=\sum_{j=1}^{M_{1}}\big\||x|^{-b}|e^{it_{n}^{j}H_{r_{n}}}\psi^{j}|^{4}\big\|_{L_{x}^{1}}+O(\varepsilon)+o_{n}(1)\\
&=\sum_{j=1}^{M}\big\||x|^{-b}|e^{it_{n}^{j}H_{r_{n}}}\psi^{j}|^{4}\big\|_{L_{x}^{1}}+\big\||x|^{-b}|W_{n}^{M_{1}}-W_{n}^{M}|^{4}\big\|_{L_{x}^{1}}+O(\varepsilon)+o_n(1)\\
&=\sum_{j=1}^{M}\big\||x|^{-b}|e^{it_{n}^{j}H_{r_{n}}}\psi^{j}|^{4}\big\|_{L_{x}^{1}}+\big\||x|^{-b}|W_{n}^{M}|^{4}\big\|_{L_{x}^{1}}+O(\varepsilon)+o_n(1).
\end{split}
\end{equation}
\end{proof} \vskip0.5cm

\begin{proposition}\label{pro5.5}
If $V$  is radial and satisfies \eqref{1.2} and \eqref{1.3}, $V\geq 0$ or $V\leq 0$, and $0<b<1$,
 there exists a radial $u_{c,0}$ in $H^1({\bf R}^{3})$ with
\begin{align}\label{5.47}
E(u_{c,0})^{s_{c}}M(u_{c,0})^{1-s_c}=\mathcal{E}_c<\mathcal{E},
\end{align}
\begin{align}\label{5.48}
\|u_{c,0}\|^{2(1-s_c)}_{L^{2}}\|H^{\frac{1}{2}}
u_{c,0}\|_{L^{2}}^{2s_{c}}<\mathcal{K},
\end{align}
such that if $u_c$ is the corresponding solution of \eqref{1.1} with
the initial data $u_{c,0}$, then $u_c$ is global and $\|u_c\|_{S(\dot{H}^{s_{c}})}=\infty$.
\end{proposition}

\begin{remark}\label{rem5.6}
When $V=0$, using the same argument as that of Proposition 6.4 in \cite{FG1}, combined with our new estimates \eqref{2.5}-\eqref{2.7} established in the present paper,
we actually extend the result obtained in \cite{FG1} to the more general case $0<b<1$, to get the following statement:
Let $V=0$ and $u_{0}\in H^{1}$ be radial and $0<b<1$. Suppose that \eqref{1.5} and \eqref{1.6} are satisfied, then the solution
$u$ of \eqref{1.1} is global in $H^{1}({\bf R}^{3})$ and scattering both forward and backward in time.

\end{remark} \vskip0.5cm

\begin{proof}
 By the assumption \eqref{5.5} and the definition of $\mathcal{E}_c$, we can
find a sequence of solutions $u_{n}(t)={\rm INLS}_{V}u_{n, 0}$ of \eqref{1.1} with initial data $u_{n,0}$ such that
\begin{align}\label{5.49}
M(u_{n,0})^{1-s_{c}}E(u_{n,0})^{s_{c}}\downarrow \mathcal{E}_c,
\end{align}
\begin{align}\label{5.50}
\| u_{n,0}\|^{2(1-s_{c})}_{L^{2}}\|H^{\frac{1}{2}}u_{n,0}\|^{2s_{c}}_{L^{2}}
<\mathcal{K}
\end{align}
and
\begin{align}\label{5.51}
\|u_n\|_{S(\dot{H}^{s_{c}})}=\infty.
\end{align}
Note that it's not obvious for uniform boundedness of $\|u_{n, 0}\|_{H^{1}}$ because of the
shortness of scaling invariance for the equation \eqref{1.1}. Hence, the first step is to show
that $\|u_{n, 0}\|_{H^{1}}$ is uniformly bounded, which can be obtained from the fact that
passing to a subsequence,
\begin{align}\label{5.52}
r_{n}=\|u_{n,0}\|_{L^{2}}^{-\frac{1}{s_{c}}}\sim 1.
\end{align}
Indeed, by the norm equivalence, we have
\begin{equation}\label{5.53}
\begin{split}
\|u_{n, 0}\|_{H^{1}}^{2}&=\|u_{n, 0}\|_{L^{2}}^{2}+\|\nabla u_{n, 0}\|_{L^{2}}^{2}\\
&\sim  \|u_{n, 0}\|_{L^{2}}^{2}+\|H^{\frac{1}{2}} u_{n, 0}\|_{L^{2}}^{2}\\
&<r_{n}^{-2s_{c}}+\mathcal{K}^{\frac{1}{s_{c}}}r_{n}^{2(1-s_{c})}.
\end{split}
\end{equation}
Now we prove \eqref{5.52} by contradiction. 
Let \eqref{5.52} be false, then we may assume that $r_{n}\rightarrow 0$ or $+\infty$. Next, we shall apply
the linear profile decomposition and the perturbation lemma to get a contradiction. To this end,
we define
$$
\tilde{u}_{n}(x, t)=\frac{1}{r_{n}^{\frac{2-b}{2}}}u_{n}(\frac{x}{r_{n}}, \frac{t}{r_{n}^{2}}),
$$
and
$$
\tilde{u}_{n, 0}(x)=\frac{1}{r_{n}^{\frac{2-b}{2}}}u_{n, 0}(\frac{x}{r_{n}}).
$$
Hence, $\tilde{u}_{n}={\rm INLS}_{V_{r_{n}}}\tilde{u}_{n, 0}$, that is, $\tilde{u}_{n}$ is the
solution to the initial value problem
\begin{equation}\label{5.54}
\left\{ \begin{aligned}
  i&\partial_{t}\tilde{u}_{n}+H_{r_{n}}\tilde{u}_{n}-|x|^{-b}|\tilde{u}_{n}|^{2}\tilde{u}_{n}=0,\\
  \;&\tilde{u}_{n}(0)=\tilde{u}_{n, 0},
                          \end{aligned}\right.
                          \end{equation}
and $\|\tilde{u}_{n, 0}\|_{H^{1}}$ is uniformly bounded, which follows from
\begin{align}\label{5.55}
\|\tilde{u}_{n, 0}\|_{L^{2}}^{2}=r_{n}^{2s_{c}}\|u_{n, 0}\|_{L^{2}}^{2}=1
\end{align}
and
\begin{equation}\label{5.56}
\begin{split}
\|\nabla\tilde{u}_{n, 0}\|_{L^{2}}^{2}
&\sim\|H_{r_{n}}^{\frac{1}{2}}\tilde{u}_{n, 0}\|_{L^{2}}^{2}
=r_{n}^{b-1}\|H^{\frac{1}{2}}u_{n, 0}\|_{L^{2}}^{2}\\
&= \|u_{n,0}\|_{L^{2}}^{\frac{2(1-s_{c})}{s_{c}}}\|H^{\frac{1}{2}}u_{n, 0}\|_{L^{2}}^{2}
<\mathcal{K}^{\frac{1}{s_{c}}}.
\end{split}
\end{equation}
Therefore, we apply Proposition \ref{pro5.2} to $\tilde{u}_{n, 0}$ to get
\begin{align}\label{5.57}
\tilde{u}_{n, 0}(x)=\sum_{j=1}^{M}e^{it_{n}^{j}H_{r_{n}}}\psi^{j}(x)+W_{n}^{M}(x).
\end{align}
Then by \eqref{5.41}, we have further that
\begin{align}\label{5.58}
\sum_{j=1}^{M}\lim_{n\rightarrow\infty}E_{V_{r_{n}}}(e^{it_{n}^{j}H_{r_{n}}}\psi^{j})
+\lim_{n\rightarrow\infty}E_{V_{r_{n}}}(W_{n}^{M})
=\lim_{n\rightarrow\infty}E_{V_{r_{n}}}(\tilde{u}_{n, 0}).
\end{align}
Since also by the profile expansion, we have
\begin{align}\label{5.59}
1=\|\tilde{u}_{n, 0}\|_{L^{2}}^{2}
=\sum_{j=1}^{M}\|\psi^{j}\|_{L^{2}}^{2}+\|W_{n}^{M}\|_{L^{2}}^{2}+o_{n}(1),
\end{align}
\begin{align}\label{5.60}
\|H_{r_{n}}^{\frac{1}{2}}\tilde{u}_{n, 0}\|_{L^{2}}^{2}
=\sum_{j=1}^{M}\|H_{r_{n}}^{\frac{1}{2}}e^{it_{n}^{j}H_{r_{n}}}\psi^{j}\|_{L^{2}}^{2}
+\|H_{r_{n}}^{\frac{1}{2}}e^{it_{n}^{j}H_{r_{n}}}W_{n}^{M}\|_{L^{2}}^{2}+o_{n}(1),
\end{align}
Since from \eqref{4.12}, each energy is nonnegative, and then
\begin{equation}\label{5.61}
\begin{split}
\lim_{n\rightarrow\infty}E_{V_{r_{n}}}(e^{it_{n}^{j}H_{r_{n}}}\psi^{j})
&\leq
\lim_{n\rightarrow\infty}E_{V_{r_{n}}}(\tilde{u}_{n, 0})
=\lim_{n\rightarrow\infty}M(u_{n,0})^{\frac{1-s_c}{s_c}}E(u_{n,0})\\
&=\mathcal{E}_c^{\frac{1}{s_{c}}}<\mathcal{E}^{\frac{1}{s_{c}}}.
\end{split}
\end{equation}
For a given $j$, if, on one hand, $|t_n^j|\rightarrow+\infty$, we may assume $t_n^j\rightarrow+\infty$
or $t_n^j\rightarrow-\infty$  up to a subsequence. In this case, by \eqref{5.59} and \eqref{5.61}
with $V=0$ and using Lemma \ref{lemFG} (iii), we have
\begin{align}\label{5.62}
\frac{1}{2}\|\nabla \psi^{j}\|_{L^{2}}\| \psi^{j}\|_{L^{2}}^{\frac{1-s_c}{s_c}}
<\mathcal{E}^{\frac{1}{s_{c}}}.
\end{align}
If we denote by ${\rm INLS}_{0}(t)\phi$   a solution of \eqref{1.1}  with $V=0$ and initial data $\phi$,
then we get  from the existence of wave operators ( Proposition \ref{pro4.3} with $V=0$ )that
there exists $\tilde{\psi}^{j}$ such that
\begin{align}\label{5.63}
\|{\rm INLS}_{0}(-t_{n}^{j})\tilde{\psi}^{j}-e^{-it_{n}^{j}\Delta}\psi^{j}\|_{H^2}
\rightarrow0,\ \ as\ \ n\rightarrow+\infty
\end{align}
and
\begin{align}\label{5.64}
\|{\rm INLS}_{0}(t)\tilde{\psi}^{j}\|_{S(\dot{H}^{s_{c}})}<\infty\;\text{and}\; \|\langle\nabla\rangle {\rm INLS}_{0}(t)\tilde{\psi}^{j}\|_{S(L^{2})}<\infty.
\end{align}
If, on the other hand, $t_{n}^{j}=0$, we set $\tilde{\psi}^{j}=\psi^{j}$.
To sum up, in either case, we obtain a $\tilde{\psi}^{j}$ for the given $\psi^{j}$ such that \eqref{5.63} and \eqref{5.64} are true.

In order to use the perturbation theory to get a contradiction, we set
$$v^{j}(t)={\rm INLS}_{0}(t)\tilde{\psi}^{j},\ \
v_{n}(t)=\sum_{j=1}^{M}v^j(t-t_{n}^{j}),\ \
\tilde{v}_{n}(t)={\rm INLS}_{0}v_{n}(0).$$ We will prove successively the following three claims
to get a contradiction.

~Claim 1. There exists a large constant $A_{0}$ and $M_{0}$ independent of $M$ such
that there exists $n_0=n_0(M)$ such that for $n\geq n_0$,
\begin{align}\label{5.65}
\|\tilde{v}_n\|_{S(\dot{H}^{s_{c}})}\leq A_{0},\quad \|\tilde{v}_n\|_{L_{t}^{\infty}H_{x}^{1}}\leq M_{0}
\end{align}
Indeed, using \eqref{5.9}, \eqref{5.63} and Lemma \ref{lemFG} (ii), we have that
\begin{align}\label{5.66}
E_{0}(v_{n}(0))=\sum_{j=1}^{M}E_{0}(v^{j}(-t^{j}))+o_{n}(1)
=\sum_{j=1}^{M}E_{0}(e^{-it_{n}^{j}\Delta}\psi^{j})+o_{n}(1)
\end{align}
By \eqref{5.22}, \eqref{5.28}, Lemma \ref{lemFG} (ii),
the assumption $r_{n}\rightarrow 0$ or $\infty$ and Lemma \ref{lem5.4} , we have
\begin{align}\label{5.67}
\begin{split}
&\quad \sum_{j=1}^{M}E_{0}(e^{-it_{n}^{j}\Delta}\psi^{j})
=\sum_{j=1}^{M}E_{V_{r_{n}}}(e^{it_{n}^{j}H_{r_{n}}}\psi^{j})+o_{n}(1)\\
&\leq E_{V_{r_{n}}}(\tilde{u}_{n, 0})+o_{n}(1)=r_{n}^{2(s_{c}-1)}E(u_{n, 0})+o_{n}(1).
\end{split}
\end{align}
Collecting \eqref{5.66} and \eqref{5.67} gives
\begin{align}\label{5.68}
E_{0}(v_{n}(0))\leq r_{n}^{2(s_{c}-1)}E(u_{n, 0})+o_{n}(1).
\end{align}
Similarly, we have
\begin{align}\label{5.69}
M(v_{n}(0))\leq M(\tilde{u}_{n, 0})+o_{n}(1)=r_{n}^{2s_{c}}M(u_{n, 0})+o_{n}(1)
\end{align}
and
\begin{align}\label{5.70}
\|\nabla v_{n}(0)\|_{L^{2}}\leq \|H_{r_{n}}^{\frac{1}{2}}\tilde{u}_{n, 0}\|_{L^{2}}+o_{n}(1)
=r_{n}^{s_{c}-1}\|H^{\frac{1}{2}}u_{n,0}\|_{L^{2}}+o_{n}(1).
\end{align}
Hence, \eqref{5.68}-\eqref{5.70} imply for large $n$,
\begin{align}\label{5.71}
M(v_{n}(0))^{1-s_{c}}E_{0}(v_{n}(0))^{s_{c}}
\leq
M(u_{n, 0})^{1-s_{c}}E(u_{n, 0})^{s_{c}}+o_{n}(1)
=\mathcal{E}_c+o_{n}(1)
<\mathcal{E}
\end{align}
and
\begin{align}\label{5.72}
\|v_{n}(0)\|_{L^{2}}^{2(1-s_{c})}\|\nabla v_{n}(0)\|_{L^{2}}^{2s_{c}}
\leq \|u_{n, 0}\|_{L^{2}}^{2(1-s_{c})}\|H^{\frac{1}{2}}u_{n,0}\|_{L^{2}}^{2s_{c}}
+o_{n}(1)<\mathcal{K}.
\end{align}
Furthermore, we claim that
\begin{align}\label{5.73}
\mathcal{E}\leq M(Q)^{1-s_{c}}E_{0}(Q)^{s_{c}}\;\text{and}\; \mathcal{K}\leq \|Q\|_{L^{2}}^{2(1-s_{c})}\|\nabla Q\|_{L^{2}}^{2s_{c}}.
\end{align}
In fact, if $V\geq 0$, \eqref{5.73} is obvious. If $V\leq 0$, it follows from the Gagliardo-Nirenberg inequality and the Pohozaev identities that
\begin{equation}\label{5.74}
\begin{split}
\frac4{(3+b)\|\mathcal{Q}\|_{L^2}^{2(1-s_{c})}\|H^{\frac{1}{2}} \mathcal{Q}\|_{L^2}^{2s_{c}} }
&=\frac{\big\||x|^{-b} |\mathcal{Q}|^{4}\big\|_{L^{1}}}{\|\mathcal{Q}\|_{L^2}^{1-b}\|H^{\frac{1}{2}} \mathcal{Q}\|_{L^2}^{3+b}}
\geq \frac{\big\||x|^{-b} |Q|^{4}\big\|_{L^{1}}}{\|Q\|_{L^2}^{1-b}\|H^{\frac{1}{2}} Q\|_{L^2}^{3+b}}\\
&\geq \frac{\big\||x|^{-b} |Q|^{4}\big\|_{L^{1}}}{\|Q\|_{L^2}^{1-b}\|\nabla Q\|_{L^2}^{3+b}}
=\frac4{(3+b)\|Q\|_{L^2}^{2(1-s_{c})}\|\nabla Q\|_{L^2}^{2s_{c}} }.
\end{split}
\end{equation}
Thus, we obtain
\begin{align}\label{5.75}
\|\mathcal{Q}\|_{L^2}^{2(1-s_{c})}\|H^{\frac{1}{2}} \mathcal{Q}\|_{L^2}^{2s_{c}}
\leq \|Q\|_{L^2}^{2(1-s_{c})}\|\nabla Q\|_{L^2}^{2s_{c}},
\end{align}
which, by \eqref{3.3} and \eqref{4.8}, implies that
\begin{equation}\label{5.76}
\begin{split}
\mathcal{E}
&=E(\mathcal{Q})^{s_{c}}M(\mathcal{Q})^{1-s_{c}}
=\Big(\frac{s_{c}}{3+b}\Big)^{s_{c}}\|\mathcal{Q}\|_{L^2}^{2(1-s_{c})}\|H^{\frac{1}{2}} \mathcal{Q}\|_{L^2}^{2s_{c}}\\
&\leq \Big(\frac{s_{c}}{3+b}\Big)^{s_{c}}\|Q\|_{L^2}^{2(1-s_{c})}\|\nabla Q\|_{L^2}^{2s_{c}}
= M(Q)^{1-s_{c}}E_{0}(Q)^{s_{c}}.
\end{split}
\end{equation}
Thus, we obtain \eqref{5.73}.

Putting together \eqref{5.71}-\eqref{5.73}, we deduce that
$$
M(v_{n}(0))^{1-s_{c}}E_{0}(v_{n}(0))^{s_{c}}
<M(Q)^{1-s_{c}}E_{0}(Q)^{s_{c}}
$$
and
$$
\|v_{n}(0)\|_{L^{2}}^{2(1-s_{c})}\|\nabla v_{n}(0)\|_{L^{2}}^{2s_{c}}
<\|Q\|_{L^{2}}^{2(1-s_{c})}\|\nabla Q\|_{L^{2}}^{2s_{c}}.
$$
Hence, it follows from Remark \ref{rem5.6} that \eqref{5.65} is true.

~Claim 2. There exists a large constant $A_{1}$ and $M_{1}$ independent of $M$ such
that there exists $n_1=n_1(M)$ such that for $n\geq n_1$,
\begin{align}\label{5.77}
\|v_n\|_{S(\dot{H}^{s_{c}})}\leq A_{1},\;\|\langle\nabla\rangle v_{n}\|_{S(L^{2})}\leq M_{1}.
\end{align}
In fact, we note that
\begin{align}\label{5.78}
i\partial_t v_n+\Delta v_n+|x|^{-b}|v_n|^{2}v_n=e_n,
\end{align}
where
\begin{align}\label{5.79}
e_n=|x|^{-b}\Big(|\sum_{j=1}^{M}v^j(t-t_{n}^{j})|^{2}\sum_{j=1}^{M}v^j(t-t_{n}^{j})-\sum_{j=1}^{M}|v^j(t-t_{n}^{j})|^{2}v^j(t-t_{n}^{j})\Big).
\end{align}
It is clear that
\begin{align}\label{5.80}
|e_n|\leq c\sum_{k\neq j}^{M}|x|^{-b}|v^j(t-t_{n}^{j})|
|v^k(t-t_{n}^{k})|^{2}.
\end{align}
Since, for $j\neq k$, $|t_{n}^{j}-
t_{n}^{k}|\rightarrow+\infty$, it follows from \eqref{2.7} and the dominated convergence theorem
that
\begin{align}\label{5.81}
\|e_n\|_{S'(\dot{H}^{-s_{c}})}\rightarrow 0\;\text{as}\; n\rightarrow\infty.
\end{align}
Next, we prove
\begin{align}\label{5.82}
\|e_n\|_{S'(L^{2})}\rightarrow 0\;\text{as}\; n\rightarrow\infty.
\end{align}
Indeed, using \eqref{5.80} again, we estimate
\begin{align}\label{5.83}
\|e_n\|_{S'(L^{2})}\leq c\sum_{k\neq j}^{M}\big\||x|^{-b}|v^j(t-t_{n}^{j})|
|v^k(t-t_{n}^{k})|^{2}\big\|_{L_{t}^{2}L_{x}^{\frac{6}{5}}}.
\end{align}
Using \eqref{5.9} and \eqref{2.6} and the dominated convergence theorem yields \eqref{5.82}.

Finally, we prove
\begin{align}\label{5.84}
\|\nabla e_n\|_{S'(L^{2})}\rightarrow 0\;\text{as}\; n\rightarrow\infty.
\end{align}
Note that
\begin{align}\label{5.85}
\nabla e_{n}=\nabla(|x|^{-b})\Big(f(v_{n})-\sum_{j=1}^{M}f(v_{j}(t-t_{n}^{j}))\Big)+|x|^{-b}\nabla\Big(f(v_{n})-\sum_{j=1}^{M}f(v_{j}(t-t_{n}^{j}))\Big)\doteq I_{1}+I_{2},
\end{align}
where $f(v)=|v|^{2}v$. For $I_{1}$.
\begin{align}\label{5.86}
\|I_{1}\|_{S'(L^{2})}\lesssim \sum_{k\neq j}^{M}\big\||x|^{-b-1}|v^j(t-t_{n}^{j})|
|v^k(t-t_{n}^{k})|^{2}\big\|_{L_{t}^{2}L_{x}^{\frac{6}{5}}}.
\end{align}
It follows from \eqref{2.13} that $\||x|^{-b-1}|v^j(t-t_{n}^{j})|
|v^k(t-t_{n}^{k})|^{2}\|_{L_{t}^{2}L_{x}^{\frac{6}{5}}}<+\infty$, and then similarly as before, we have
\begin{align}\label{5.87}
\|I_{1}\|_{S'(L^{2})}\rightarrow 0\;\text{as}\; n\rightarrow\infty.
\end{align}
For $I_{2}$, note that
\begin{align}\label{5.88}
\|I_{2}\|_{S'(L^{2})}\lesssim\sum_{k\neq j}^{M}\Big\||x|^{-b}|v^k(t-t_{n}^{j})|(|v^j(t-t_{n}^{j})|+|v^k(t-t_{n}^{j})|)
|\nabla v^j(t-t_{n}^{k})|^{2}\Big\|_{L_{t}^{2}L_{x}^{\frac{6}{5}}}.
\end{align}
From the proof of \eqref{2.16} and similarly as before, we deduce that
\begin{align}\label{5.89}
\|I_{2}\|_{S'(L^{2})}\rightarrow 0\;\text{as}\; n\rightarrow\infty.
\end{align}
Putting \eqref{5.87} and \eqref{5.89} together gives \eqref{5.84}.

Applying \eqref{5.81}, \eqref{5.82} and \eqref{5.84} with \eqref{5.65} to Lemma \ref{lem2.9} with $V=0$, gives \eqref{5.77}.

~Claim 3. There exists a large constant $A_{2}$ independent of $M$ such
that there exists $n_2=n_2(M)$ such that for $n\geq n_2$,
\begin{align}\label{5.90}
\|\tilde{u}_n\|_{S(\dot{H}^{s_{c}})}\leq A_{2}.
\end{align}
To see this, we note that
\begin{align}\label{5.91}
i\partial_t v_n-H_{r_{n}} v_n+|x|^{-b}|v_n|^{2}v_n={\tilde e}_n,
\end{align}
where
\begin{align}\label{5.92}
{\tilde e}_n=e_{n}-V_{r_{n}}v_{n}.
\end{align}
We will use the perturbation theory to get \eqref{5.90}. To this end, we should control the following four terms, that is,
\begin{align}\label{5.93}
\|\tilde{u}_{n, 0}-v_{n}(0)\|_{H^{1}},\;\;\|e^{-itH_{r_{n}}}(\tilde{u}_{n, 0}-v_{n}(0))\|_{S(\dot{H}^{s_{c}})},
\end{align}
\begin{align}\label{5.94}
\|{\tilde e}_n\|_{S'(\dot{H}^{-s_{c}})},\;\;
\|\langle\nabla\rangle e_{n}\|_{S(L^{2})}.
\end{align}

From \eqref{5.57} and the definition of $v_{n}(t)$, we have
\begin{align}\label{5.95}
\tilde{u}_{n, 0}-v_{n}(0)=W_{n}^{M}+\sum_{j=1}^{M}(e^{it_{n}^{j}H_{r_{n}}}\psi^{j}-v^{j}(-t_{n}^{j})).
\end{align}
As $\|\tilde{u}_{n, 0}\|_{H^{1}}$ is uniformly bounded,
\begin{align}\label{5.96}
\|W_{n}^{M}\|_{H^{1}}\;\text{ is uniformly bounded too.}
\end{align}
From the triangle inequality, \eqref{5.22}, \eqref{5.27} and \eqref{5.63}, it follows that
\begin{align}\label{5.97}
\|e^{it_{n}^{j}H_{r_{n}}}\psi^{j}-v^{j}(-t_{n}^{j})\|_{H^{1}}\rightarrow 0\;\text{as}\; n\rightarrow 0,
\end{align}
which combined with \eqref{5.96} implies that
\begin{align}\label{5.98}
\|\tilde{u}_{n, 0}-v_{n}(0)\|_{H^{1}}\;\text{ is uniformly bounded.}
\end{align}

 Let $\epsilon_0=\epsilon_0(A_{2},n)$ be
a small number given in Lemma \ref{lem2.9}. By \eqref{5.10}, takeing $M$ large enough such that
there exists $n_{3}=n_{3}(M)$ satisfying
\begin{align}\label{5.99}
\|e^{-itH_{r_{n}}}W_{n}^{M}\|_{S(\dot{H}^{s_{c}})}< \frac{\epsilon_{0}}{2}
\end{align}
for all $n\geq n_{3}$. Next we turn to the estimate of
\begin{align}\label{5.100}
\|e^{-itH_{r_{n}}}(e^{it_{n}^{j}H_{r_{n}}}\psi^{j}-v^{j}(-t_{n}^{j}))\|_{S(\dot{H}^{s_{c}})}
\end{align}
for each $j$.
From Strichartz estimates and \eqref{5.97}, it follows that
there exists $n_{4}=n_{4}(M)$ such that for each $j$ and $n\geq n_{4}$
\begin{align}\label{5.101}
\|e^{itH_{r_{n}}}(e^{it_{n}^{j}H_{r_{n}}}\psi^{j}-v^{j}(-t_{n}^{j}))\|_{S(\dot{H}^{s_{c}})}< \frac{\epsilon_{0}}{2M}.
\end{align}
~From \eqref{5.99} and \eqref{5.101}, it follows that
\begin{align}\label{5.102}
\|e^{itH_{r_{n}}}(\tilde{u}_{n, 0}-v_{n}(0))\|_{S(\dot{H}^{s_{c}})}<\epsilon_{0}
\end{align}
for all $n\geq\max\{ n_{3}, n_{4}\}$.

Similar to the proof of \eqref{5.18}, \eqref{5.22} and \eqref{5.27}, by using \eqref{5.77}, we have that both
$\|V_{r_{n}}v_{n}\|_{S'(\dot{H}^{-s_{c}})}$ and $\|\langle\nabla\rangle (V_{r_{n}}v_{n})\|_{S(L^{2})}$ go to zero as $n\rightarrow\infty$,
which together with $\lim_{n\rightarrow\infty}\|e_n\|_{S'(\dot{H}^{-s_{c}})}=0$ and $\lim_{n\rightarrow\infty}\|\langle\nabla\rangle e_n\|_{S'(L^{2})}=0$ gives
\begin{align}\label{5.103}
\lim_{n\rightarrow\infty}\|\tilde{e}_n\|_{S'(\dot{H}^{-s_{c}})}=\lim_{n\rightarrow\infty}\|\langle\nabla\rangle \tilde{e}_n\|_{S'(L^{2})}=0.
\end{align}
Applying Lemma \ref{lem2.9} with \eqref{5.98}, \eqref{5.102}, \eqref{5.103} and \eqref{5.77}, we get \eqref{5.90}.

By scaling, we have
\begin{align}\label{5.104}
\|u_n\|_{S(\dot{H}^{s_{c}})}=\|\tilde{u}_n\|_{S(\dot{H}^{s_{c}})}\leq A_{2},
\end{align}
 contradicting with \eqref{5.51}. So $\|u_{n, 0}\|_{H^{1}}$ is uniformly bounded.

The next step is to extract $u_{c,0}$ from a bounded sequence $\{u_{n, 0}\}_{n=1}^{+\infty}$.
We omit the proof because it is similar to the proof of Proposition 6.4 in \cite{FG1}. Indeed, it
suffices to replace $e^{-it\Delta}$ and $\nabla$ by $e^{-itH}$ and $H^{\frac{1}{2}}$  respectively in the above proof. In addition,
we need to use the new estimates \eqref{2.5}-\eqref{2.7} (see the proof of \eqref{5.81}, \eqref{5.82} and \eqref{5.84}).
\end{proof}

Once we established Proposition \ref{pro5.5}, we can obtain the following results of precompactness and
uniform localization of the minimal blow-up solution, the proof of which is standard and we omit it
here.

\begin{proposition}\label{pro5.8}
 Let $u_c$ be
as  in Proposition \ref{pro5.5}. Then
$$
K=\{u_c(t)| ~t\in{\bf R}\}\subset H^1({\bf R}^{3})
$$
is precompact in $H^1({\bf R}^{3})$.
\end{proposition}

\begin{coro}\label{cor5.9}
 Let $u$ be a solution of \eqref{1.1} such that $K=\{u(t)|
~t\in {\bf R} \}$ is precompact in $H^1({\bf R}^{3})$. Then for each
$\epsilon>0,$ there exists $R>0$ independent of $t$ such that
\begin{align}\label{5.105}
\int_{|x|>R}|\nabla u(x,t)|^2+|u(x,t)|^2+|u(x,t)|^4+|u(x,t)|^{6}dx\leq\epsilon.
\end{align}
\end{coro}

\medskip

\section{Scattering result}
\setcounter{equation}{0}

In this section, we prove the following rigidity statement and finish the proof of Theorem \ref{th1.2}.

\begin{theorem}\label{th6.1}
If $V$ is radial and satisfies \eqref{1.2} and \eqref{1.3}, $x\cdot\nabla V\leq0$, $|x|\cdot|\nabla V|\in L^{\frac{3}{2}}$, and
$0<b<1$.
 Suppose that $u_0\in H^1({\bf R}^{3})$ is radial,
$M(u_{0})^{1-s_{c}}E(u_0)^{s_{c}}<\mathcal{E}$ and
$\|u_{0}\|^{1-s_{c}}_{L^{2}}\|H^{\frac{1}{2}} u_{0}\|_{L^{2}}^{s_{c}}<\mathcal{K}.$ Let $u$
be the corresponding solution of \eqref{1.1} with initial data
$u_0$. If $K_+=\{u(t):t\in[0,\infty)\}$ is precompact in $H^1({\bf R}^{3})$, then
$u_0\equiv0$. The same conclusion holds if
$K_-=\{u(t):t\in(-\infty,0]\}$ is precompact in $H^1({\bf R}^{3})$.
\end{theorem}

\begin{proof}
By Theorem \ref{th1.1}, we have that $u$ is global in $H^{1}({\bf R}^{3})$ and
\begin{align}\label{6.1}
\|u(t)\|^{1-s_{c}}_{L^{2}}\|H^{\frac{1}{2}} u(t)\|_{L^{2}}^{s_{c}}<\|Q\|^{1-s_{c}}\|\nabla Q\|_{L^{2}}^{s_{c}}
\end{align}

 We first define
\begin{align}\label{6.2}
M_{a}(t)=2\displaystyle\int_{{\bf R}^{3}}\partial_{j}a \Im(\bar{u}\partial_{j}u)dx,
\end{align}
where $a\in C_{c}^{\infty}({\bf R}^{3})$ and we always take summation with respect to repeated indices. 
Following the computation of Lemma 5.3 in Tao, Visan and Zhang \cite{TVZ1} (see also Lemma 4.1 in \cite{Dinh1} ) yields
\begin{equation}\label{6.3}
\begin{split}
M_{a}'(t)
&=2\displaystyle\int_{{\bf R}^{3}}\Big(2\partial_{jk} a \Re(\partial_{j}\bar{u}\partial_{k}u)
-\frac{1}{2}\Delta^{2}a|u|^{2}\Big)dx\\
&\quad-\displaystyle\int_{{\bf R}^{3}}\Delta a |x|^{-b}|u|^{4}dx+\displaystyle\int_{{\bf R}^{3}}\nabla a\cdot\nabla(|x|^{-b})|u|^{4}dx\\
&\quad-2\displaystyle\int_{{\bf R}^{3}}\nabla a\cdot\nabla V|u|^{2}dx,
\end{split}
\end{equation}
Take a radially symmetric function $\phi\in C_{c}^{\infty}$ such that
$\phi(x)=|x|^{2}$ for $|x|\leq 1$ and $\phi(x)=0$ for $|x|\geq 2$, and define
$a(x)=R^{2}\phi(\frac{x}{R})$. 
By the repulsiveness assumption on the potential $V$ (i.e., $x\cdot \nabla V\leq 0$), direct computation gives
\begin{equation}\label{6.4}
\begin{split}
M_{a}'(t)
&=8\displaystyle\int_{{\bf R}^{3}}|\nabla u|^{2}dx
-2(3+b)\displaystyle\int_{{\bf R}^{3}}|x|^{-b}|u|^{4}dx
-4\displaystyle\int_{{\bf R}^{3}}x\cdot\nabla V|u|^{2}dx
+({\rm Eror})\\
&\geq  8\displaystyle\int_{{\bf R}^{3}}|\nabla u|^{2}dx
-2(3+b)\displaystyle\int_{{\bf R}^{3}}|x|^{-b}|u|^{4}dx
+({\rm Error}),
\end{split}
\end{equation}
where
\begin{equation}\label{6.5}
\begin{split}
{\rm (Erro)}&=4\Re\displaystyle\int_{{\bf R}^3}\Big(\partial_{j}^{2}\phi(\frac{x}{R})-2\Big)|\partial_{j}u|^{2}dx
+4\sum_{j\neq k}\Re\displaystyle\int_{{\bf R}^3}(\partial_{jk}\phi)(\frac{x}{R})\partial_{k}u\partial_{j}\bar{u}dx\\
&\quad-\frac{1}{R^{2}}\displaystyle\int_{{\bf R}^3}(\Delta^{2}\phi)(\frac{x}{R})|u|^{2}dx
+R\displaystyle\int_{{\bf R}^3}\nabla(|x|^{-b})\cdot(\nabla \phi)(\frac{x}{R})|u|^{4}dx\\
&\quad+\displaystyle\int_{{\bf R}^3}\Big(-\Big(-\Delta\phi(\frac{x}{R})-6\Big)+2b\Big)|x|^{-b}|u|^{4}dx\\
&\quad+2\displaystyle\int_{{\bf R}^3}\Big(2x\cdot\nabla V-R(\nabla\phi)(\frac{x}{R})\nabla V\Big)|u|^{2}dx.
\end{split}
\end{equation}
By the definition of $\phi(x)$ and
\begin{align}\label{6.6}
2\displaystyle\int_{|x|\leq R}\nabla(|x|^{-b})\cdot x |u|^{4}dx=2\displaystyle\int_{|x|\leq R}-b|x|^{-b}|u|^{4}dx,
\end{align}
it follows from Corollary \ref{cor5.9} that $({\rm Error})\rightarrow 0$ as $R\rightarrow\infty$ uniformly in
$t\in [0,\infty)$. In fact,
\begin{align}\label{6.7}
({\rm Error})
&\lesssim \displaystyle\int_{|x|\geq R}|\nabla u|^{2}dx
+\displaystyle\int_{|x|\geq R}|x|^{-b}|u|^{4}dx
+\frac{1}{R^{2}}\displaystyle\int_{|x|\geq R}|u|^{2}dx
+\big\||x||\nabla V|\big\|_{L^{\frac{3}{2}}}\|u\|_{L^{6}(|x|\geq R)}^{2}\nonumber\\
&+\displaystyle\int_{|x|\geq R}|\nabla u|^{2}dx
+\displaystyle\int_{|x|\geq R}\frac{1}{R^{b}}|u|^{4}dx
+\frac{1}{R^{2}}\displaystyle\int_{|x|\geq R}|u|^{2}dx
+\big\||x||\nabla V|\big\|_{L^{\frac{3}{2}}}\|u\|_{L^{6}(|x|\geq R)}^{2}
\rightarrow 0.
\end{align}

Putting \eqref{6.4}, \eqref{6.7}, \eqref{4.14}and \eqref{4.12} together and using the norm equivalence yield that
there exists some constant $\delta_{0}>0$ such that
\begin{align}\label{6.8}
M_{a}'(t)\geq \delta_0\displaystyle\int_{{\bf R}^{3}}|\nabla u_0|^2dx.
\end{align}
Thus, we have
\begin{align}\label{6.9}
M_{a}(0)-M_{a}(t)\geq \delta_0 t\displaystyle\int_{{\bf R}^{3}}|\nabla u_0|^2dx.
\end{align}
On the other hand, by the definition of $M_{a}(t)$, the norm equivalence and \eqref{6.1}, we should have
\begin{align}\label{6.10}
|M_{a}(t)|
\leq R\|u\|_{L^{2}}\|\nabla u\|_{L^{2}}
\lesssim R\|u\|_{L^{2}}\|H^{\frac{1}{2}} u\|_{L^{2}}
\leq cR,
\end{align}
which is a contradiction for $t$ large unless $u_{0}=0$.
\end{proof}
 \vskip0.5cm

Now, we can finish the proof of Theorem \ref{th1.2}.\\
{\bf The  Proof of Theorem \ref{th1.2}.}
In view of Proposition \ref{pro5.8}, Theorem \ref{th6.1}
implies that $u_c$ obtained in Proposition \ref{pro5.5} cannot
exist. Thus, there must holds that $\mathcal{E}_c
=\mathcal{E}$, which combined with
Proposition \ref{pro2.6} implies Theorem \ref{th1.2}.
$\hfill\Box$ \vskip0.5cm

\section{Blow-up criteria}
 We finally consider the blow-up in finite or infinite time following the idea from Du-Wu-Zhang \cite{DWZ}.

{\bf Proof of Theorem \ref{the2}}
Assume the contrary, then we have 
 $$ C_0=\sup_{t\in\mathbb R^+}\|\nabla u(t)\|_{L^2}<\infty.$$
Consider the local Virial identity and let
\begin{equation}\label{vf}
I(t)=\int_{{\bf R}^{3}}\phi(x)|u(t,x)|^2dx,
\end{equation}
where $\phi\in C_{c}^{\infty}({\bf R}^{3})$. Similar to \eqref{6.3}, we get
$$I'(t)=2\Im\int_{{\bf R}^{3}}\nabla\phi\cdot\nabla u\bar udx;$$
\begin{align*}
I''(t)=\int_{{\bf R}^{3}}4\Re\nabla\bar u\nabla^2\phi\nabla udx
-\int_{{\bf R}^{3}}2\nabla\phi\cdot\nabla V|u|^2+\Delta\phi |x|^{-b}|u|^4-\nabla\phi\cdot\nabla(|x|^{-b})|u|^4dx
-\int_{{\bf R}^{3}}\Delta^2\phi|u|^2dx.
\end{align*}
In particular, if $\phi$ is radial, 
\begin{align}\label{I'}
I'(t)=2\Im\int_{{\bf R}^{3}}\phi'(r)\frac{x\cdot\nabla u}r\bar udx
\end{align}
and
\begin{align*}
&I''(t)=4\int_{{\bf R}^{3}}\frac{\phi'}r|\nabla u|^2dx+4\int_{{\bf R}^{3}}\left(\frac{\phi''}{r^2}-\frac{\phi'}{r^3}\right)|x\cdot\nabla u|^2dx\\
&-2\int_{{\bf R}^{3}}\frac{\phi'}{r}x\cdot\nabla V|u|^2dx-\int_{{\bf R}^{3}}\left(\phi''(r)+\frac{2+b}r\phi'(r)\right) |x|^{-b}|u|^4dx
-\int_{{\bf R}^{3}}\Delta^2\phi|u|^2dx.
\end{align*}

{\bf $L^2$ estimate in the exterior ball}
\begin{lemma}\label{leml2}
Given $\eta_0>0$, then for any 
$$
0<t\leq\frac{\eta_0R}{4m_0 C_0},
$$
$$
\int_{|x|\geq\frac R2}|u(t,x)|^2dx\leq\eta_0+o_R(1).
$$
\end{lemma}

\begin{proof}
Take the radial function $\phi$ in \eqref{vf} such  that
$$
\phi=\begin{cases}0,&0\leq r\leq\frac R2;\\1,&r\geq R,\end{cases}
$$
and  
$$
0\leq\phi\leq1,\ \ 0\leq\phi'\leq\frac4R,
$$
where $R>0$ is a large constant and will be chosen later.
By \eqref{I'}, there holds that
\begin{align*}
I(t)=&I(0)+\int_0^tI'(\tau)d\tau
\leq I(0)+t\|\phi'\|_{L^\infty}\sup_{s\in[0,t]}(\|u(s)\|_{L^2}\|\nabla u(s)\|_{L^2})\\
\leq&\int_{|x|\geq\frac R2}|u_0|^2dx+\frac{4m_0 C_0t}R=o_R(1)+\frac{4m_0 C_0t}R
\end{align*}
where $m_0=\|u_0\|_{L^2}$ and $o_R(1)$ tends to 0 as $R\rightarrow+\infty$. 
By the definition of $\phi$, 
$$
\int_{|x|\geq  R}|u(t,x)|^2dx\leq I(t).
$$
To sum up, we complete the proof of the lemma.
\end{proof}

{\bf Localized Virial identity}

At this stage, we choose $\phi$ such that
$$0\leq\phi\leq r^2,\ \ \phi''\leq2,\ \ \phi^{(4)}\leq\frac4{R^2},$$
and
$$\phi=\begin{cases}r^2,&0\leq r\leq R;\\0,&r\geq 2R\end{cases}$$
to get the following result.
\begin{lemma}\label{lemI''}
There exist two constant $\tilde C=\tilde C(m_0, C_0)>0$, $\theta_0>0$, such that
$$I''(t)\leq8K(u(t))+\tilde C\|u\|_{L^2(|x|>R)}^{\theta_0}.$$
\end{lemma}

\begin{proof}
By the definition of $K(u)$ \eqref{K},
$$
I''(t)=8K(u)+R_1+R_2+R_3+R_4,
$$
where
$$R_1=4\int_{{\bf R}^{3}}\left(\frac{\phi'}r-2\right)|\nabla u|^2dx+4\int_{{\bf R}^{3}}
\left(\frac{\phi''}{r^2}-\frac{\phi'}{r^3}\right)|x\cdot\nabla u|^2dx,$$
$$R_2=-\int_{{\bf R}^{3}}\left(\phi''+\frac{2-b}r\phi'(r)-(6-2b)\right)|x|^{-b}| u|^4dx,$$
$$R_3=-2\int_{{\bf R}^{3}}\left(\frac{\phi'}{r}-2\right)(x\cdot\nabla V)|u|^2dx,$$
$$R_4=-\int_{{\bf R}^{3}}\Delta^2\phi|u|^2dx.$$
We claim that $R_1$ and $R_{3}$ are non-positive and $R_2$ and $R_4$ are the error terms.

Indeed, for $R_{1}$, 
we divide ${\bf R}^3$ of the second term into two parts: 
$$
A=\left\{x\in {\bf R}^3: \frac{\phi''}{r^2}-\frac{\phi'}{r^3}\leq0\right\}\;\;\text{and}\;\;
B=\left\{x\in {\bf R}^3: \frac{\phi''}{r^2}-\frac{\phi'}{r^3}>0\right\}.
$$
Correspondingly, $R_{1}=R_{1}^{A}+R_{1}^{B}$.
For $R_{1}^{A}$, 
it is trivial that $R_{1}^{A}\leq 0$ since $\phi'\leq2r$.
For $R_{1}^{B}$, since $\phi''\leq2$, it holds that
$$R_1^{B}\leq 4\int_{{\bf R}^{3}}\left(\frac{\phi'}r-2\right)|\nabla u|^2dx+4\int_{{\bf R}^{3}}
\left(2-\frac{\phi'}r\right)|\nabla u|^2dx=0.$$
Hence, $R_{1}\leq 0$.
For $R_{2}$, since 
$$
{\rm supp}\Big\{r\in [0,+\infty):\phi''+\frac{2-b}r\phi'(r)-(6-2b)\Big\}\subset[R,\infty),
$$ 
by interpolation and Sobolev embedding, we have
$$
R_2\lesssim R^{-b}\|u\|_{L^{4}(|x|>R)}^{4}\lesssim R^{-b}\|u\|_{L^6(|x|>R)}^{3}\|u\|_{L^2(|x|>R)}
\lesssim R^{-b}\|\nabla u\|_{L^{2}}^{3}\|u\|_{L^2(|x|>R)}\lesssim C_{0}^{3}R^{-b}\|u\|_{L^2(|x|>R)}.
$$
For $R_{3}$, by the assumption $x\cdot\nabla V\leq0$ \eqref{Ve1},  we obtain $R_3\leq0$.
Finally, for $R_{4}$, by the properties of $\phi$, $R_4\lesssim R^{-2}\|u\|_{L^2(|x|>R)}^2$.

Putting all the above estimates together, there holds that for $R>1$,
$$I''(t)\leq8K(u)+\tilde C\|u\|_{L^2(|x|>R)},$$
where $\tilde C>0$ depending on $m_0$ and $C_0$. Thus, we complete the proof of the lemma.
\end{proof}

Applying Lemma \ref{leml2} and \ref{lemI''}, we have that for any
$t\leq T:=\eta_0R/(4m_0 C_0)$,
$$I''(t)\leq8K(u)+\tilde C(\eta_0^{1/2}+o_R(1)).$$
Integrating over the interval $[0, T]$ and using the assumption \eqref{Kbeta} gives that
\begin{align*}
I(T)&\leq I(0)+I'(0)T+\int_0^T\int_0^t(8K(u(s))+\tilde C\eta_0^{1/2}+o_R(1))\\
&\leq I(0)+I'(0)T+(8\beta_0+\tilde C\eta_0^{1/2}+o_R(1))\frac{T^2}2.
\end{align*}
If $\eta_0=\beta_{0}^{2}\tilde{c}^{-2}$ 
and $o_{R}(1)<-\beta_{0}$ (by taking $R$ large enough), then for $T=\eta_0R/(4m_0 C_0)$, we obtain that
\begin{align}\label{IT}
I(T)\leq I(0)+I'(0)\eta_0R/(4m_0 C_0)+\alpha_0R^2,
\end{align}
where the constant $$\alpha_0=\beta_0\eta_0^2/(4m_0 C_0)^2<0$$ is independent of $R$.

Note that
\begin{align}\label{I0}
I(0)=o_R(1)R^2,\ \ \ I'(0)=o_R(1)R.
\end{align}
In fact,
\begin{align*}
I(0)&\leq\int_{|x|<\sqrt{R}}|x|^2|u_0|^2dx+\int_{\sqrt{R}<|x|<2R}|x|^2|u_0|^2dx\\
&\leq Rm_0^2+4R^2\int_{|x|>\sqrt{R}}|u_0|^2dx=o_R(1)R^2.
\end{align*}
The argument gives the second estimate.

Substituting \eqref{I0} in \eqref{IT} and taking $R$ large enough, we have that
$$
I(T)\leq o_R(1)R^2+\alpha_0R^2\leq\frac12\alpha_0R^2<0,
$$
which contradicts with the definition of $I$.
Thus, we complete the proof of Theorem \ref{the2}.

$\hfill\Box$

We finally finish the proof of Theorem \ref{the1}.

{\bf Proof of Theorem \ref{the1}}
Firstly, we need to show that
the assumption \eqref{4.1} and \eqref{4.4} implies \eqref{Kbeta}. Indeed,
first by Theorem \ref{th4.1}, we know that \eqref{4.5} holds for any $t\in[0,T_{max})$.
Then by the definition of $K(u)$ \eqref{K}, we get
\begin{align}\label{K1}
K(u)=(3+b)E(u)-\frac{1+b}2\|H^{\frac12}u\|_{L^2}^2-\frac12\int_{{\bf R}^{3}}(2V+x\cdot\nabla V)|u|^2dx.
\end{align}
Then by \eqref{4.5} and the assumption $2V+x\cdot\nabla V\geq0$ in \eqref{Ve1}, one obtains that
$$K(u(t)<0,\ \ for\ any\ t\in[0,T_{max}).$$
Now we claim that there exists some $\delta_0>0$ such that for any $t\in[0,T_{max})$,
\begin{align}\label{K2}
K(u)<-\delta_0\|H^{\frac12}u\|_{L^2}^2.
\end{align}
Indeed, if on the contrary, there exists some time sequence $\{t_n\}\subset[0,T_{max})$ such that
$$-\delta_n\frac{1+b}2\|H^{\frac12}u\|_{L^2}^2<K(u(t_n))<0,$$
where $\delta_n\rightarrow0$ as $n\rightarrow\infty$. Then by \eqref{K1},
$$E(u(t_n))\geq\frac1{3+b}\left(K(u(t_n))+\frac{1+b}2\|H^{\frac12}u\|_{L^2}^2
\right)>(1-\delta_n)\frac{1+b}{2(3+b)}\|H^{\frac12}u\|_{L^2}^2.$$
Therefore, we obtain
\begin{align*}
&M(u(t_n))^{1-s_c}E(u(t_n))^{s_c}\\
&>M(u(t_n))^{1-s_c}(1-\delta_n)^{s_c}\left(\frac{1+b}{2(3+b)}\right)^{s_c}
\|H^{\frac12}u\|_{L^2}^{2s_c}\\
&>(1-\delta_n)^{s_c}\left(\frac{1+b}{2(3+b)}\right)^{s_c}\mathcal K=(1-\delta_n)^{s_c}\mathcal E,
\end{align*}
 contradicting \eqref{4.1} and \eqref{K2} holds.
 Finally, since by \eqref{4.5}, the Kinetic $\|H^{\frac12}u\|_{L^2}^2>\epsilon_0$ with some
 positive constant $\epsilon_0>0$, then we immediately obtain \eqref{Kbeta} and Theorem \ref{the1} is proved by Theorem \ref{the2}.
 $\hfill\Box$


\begin{thebibliography}{99}
 \bibitem{BCD} V. Banica, R. Carles and T. Duyckaerts, Minimal blow-up solutions to the mass-critical inhomogeneous
 NLS equation, \emph{ Comm. Partial Differential Equations}, 36(2010), pp. 487-531.

 \bibitem{Caz} T. Cazenave, \emph{Semilinear Schr\"{o}dinger equations}, Courant Lecture Notes in
Mathematics, 10. New York University, Courant Institute of Mathematical Sciences, New York; American
Mathematical Society, Providence, RI, 2003. Xiv+323 PP. ISBN£º0-8218-3399-5.

 \bibitem{CG} V. Combet and F. Genoud, Classification of minimal mass blow-up solutions for an $L^{2}$ critical inhomogeneous
 NLS, \emph{J. Evol. Equ.},
 16(2016), pp. 483-500.

 \bibitem{CG1} E. Csobo and F. Genoud, Minimal mass blow-up solutions for the $L^{2}$ critical NLS with inverse-square potential,
 arXiv:1707.01421.

 \bibitem{Dinh1} V. D. Dinh, Scattering theory in a weighted $L^{2}$ space for a class of the defocusing inhomogeneous nonlinear
 Schr\"{o}dinger equation, arXiv: 1710.01392.

  \bibitem{Dinh2} V. D. Dinh, Energy scattering for a class of the defocusing inhomogeneous nonlinear Schr\"{o}dinger equation,
  arXiv:1710.05766.

  \bibitem{Dinh3} V. D. Dinh, Blow up of $H^{1}$ solutions for a class of the focusing inhomogeneous nonlinear Schr\"{o}dinger equation,
  arXiv:1711.09088.

\bibitem{DWZ} D. Du, Y. Wu and K. Zhang, On blow-up criterion for the nonlinear Schr\"{o}dinger equation, \emph{ Discrete Contin. Dyn. Syst.},
36(2016), pp. 3639-3650.

\bibitem{DHR} T. Duyckaerts, J. Holmer and S. Roudenko, Scattering for the non-radial 3D cubic
nonlinear Schr\"{o}dinger equation, \emph{ Math. Res. Lett.}, 15(2008), pp. 1233-1250.

\bibitem{Far} L. G. Farah, Global well-posedness and blow-up on the energy space for the inhomogeneous nonlinear Schr\"{o}dinger
equation, \emph{J. Evol. Equ.}, 16(2016), pp. 193-208.

\bibitem{FG1} L. G. Farah and G. M. Guzm$\acute{\rm{a}}$n, Scattering for the radial 3D cubic focusing inhomogeneous
nonlinear Schr\"{o}dinger equation, arXiv:1610.06523.

\bibitem{FG2} L. G. Farah and G. M. Guzm$\acute{\rm{a}}$n, Scattering for the radial focusing INLS equation in higher
dimensions, arXiv:1703.10988.

 \bibitem{Fos} D. Foschi, Inhomogeneous Strichartz estimates, \emph{ J. Hyperbolic Differ. Equ.}, 2(2005), pp. 1-24.


\bibitem{Gen} F. Geneoud, An inhomogeneous, $L^{2}$-critical, nonlinear Schr\"{o}dinger equation, \emph{ Z. Anal. Anwend.},
31(2012), pp. 283-290.

 \bibitem{GS} F. Genoud and C. A. Stuart, Schr\"{o}dinger equations with a spatially decaying nonlinearity: existence and
 stability of standing waves, \emph{ Discrete Contin. Dyn. Syst.}, 21(2008), pp. 137-186.

 \bibitem{Guz} C. M. Guzm$\acute{{\rm a}}$n, On well posedness for the inhomogeneous nonlinear Schr\"{o}dinger equation,
 arXiv:1606.02777.

 \bibitem{HK} T. Hmidi and S. Keraani, Blowup theory for the critical nonlinear Schr\"{o}dinger equation revisited,
 \emph{Int. Math. Res. Not.}, 46(2005), pp. 2815-2828.

 \bibitem{HR} J. Holmer and S. Roudenko, A sharp condition for scattering of the radial 3D cubic
nonlinear Schr\"{o}dinger equation, \emph{ Comm. Math. Phy.}, 282(2008), pp. 435-467.

 \bibitem{H} Y. Hong, Scattering for a nonlinear Schr\"{o}dinger equation with a potential,
\emph{Commun. Pure Appl. Anal.}, 15(2016), pp. 1571-1601.


 \bibitem{IMN} S. Ibrahim, N. Masmoudi and K. Nakanishi, Scattering threshold for the focusing nonlinear Klein-Gordon equation, \emph{
 Anal. PDE}, 4(2011), 405-460.


 \bibitem{Keel} M. Keel and T. Tao, Endpoint Strichartz estimates, \emph{ Amer. J. Math.}, 120(1998), pp. 955-980.

 \bibitem{Kenig} C. Kenig and F. Merle, Global well-posedness, scattering and blow-up for the energy-critical,
focusing, non-linear Schr\"{o}dinger equaiton in the radial case, \emph{ Invent. Math.}, 166(2006), pp. 645-675.

 \bibitem{Ker} S. Keraani, On the defect of compactness for Strichartz estimates of the Schr\"{o}dinger
equations, \emph{ J. Differential Equations}, 175(2001), pp. 353-392.

\bibitem{KMVZ} R. Killip, J. Murphy, M. Visan and J. Zheng, The focusing cubic NLS with inverse-square potential in three
space dimensions, arXiv: 1603.08912.



 \bibitem{M} F. Merle, Nonexistence of minimal blow-up solutions of equations $iu_{t}=-\Delta u-k(x)|u|^{\frac{4}{N}}u$
 in ${\bf R}^{N}$, \emph{ Ann. Inst. H. Poincar$\acute{e}$ Phys. Th$\acute{e}$or.}, 64(1996), pp. 33-85.



 \bibitem{RS} P. Rapha\"{e}l and J. Szeftel, Existence and uniqueness of minimal blow-up solutions
 to an inhomogeneous mass critical MLS, \emph{ J. Amer. Math. Soc.}, 24(2011), pp. 471-546.


\bibitem{Sch} W. Schlag, \emph{ Notes from a course on Schr\"{o}dinger equaitons}, https://math.uchicago.edu/~schlag/.



 \bibitem{SW} E. Stein and G. Weiss, Fractional integrals on $n$-dimensional Euclidean space,
  \emph{ J. Math. Mech.}, 7(1958),
 pp. 503-514.

 \bibitem{TVZ1} T. Tao, M. Visan and X. Zhang, The nonlinear Schr\"{o}dinger equation with combined power-type
 nonlinearities, \emph{ Comm. Partial Differential Equations}, 32(2007), pp. 1281-1343.


 \bibitem{GWY} Q. Guo, H. Wang and X. Yao, Dynamics of the focusing 3D cubic NLS with slowly decaying potential,
 arXiv: 1811.07578.




\end{thebibliography}
\end{document}